\documentclass[11pt, reqno, a4paper]{amsart}
\usepackage{iftex}

\ifPDFTeX
  \usepackage[T1]{fontenc}
  \usepackage{lmodern}
  \usepackage{xcolor}
  \usepackage{graphicx}
  \usepackage{tikz}
  \usepackage{hyperref}
\else
  \usepackage[T1]{fontenc}
  \usepackage{lmodern}
  \usepackage[dvipdfmx]{xcolor}
  \usepackage[dvipdfmx]{graphicx}
  
  \usepackage{tikz}
  \usepackage[dvipdfmx]{hyperref}
\fi

\definecolor{darkblue}{rgb}{0.0, 0.0, 0.55}
\definecolor{darkmagenta}{rgb}{0.55, 0.0, 0.55}
\definecolor{darkcyan}{rgb}{0.0, 0.55, 0.55}

\hypersetup{
  colorlinks = true,
  linkcolor  = darkblue,
  citecolor  = darkmagenta,
  urlcolor   = darkcyan
}

\usepackage{amssymb, amsthm}
\usepackage{mathtools}
\usepackage{mathrsfs}
\usepackage[g]{esvect}
\usepackage{booktabs}
\usepackage{stmaryrd}

\usepackage[
  left=2.5cm,
  right=2.5cm,
  top=2.5cm,
  bottom=2.5cm,
]{geometry}

\usepackage{enumitem}
\setlist[enumerate]{
  label=\textup{(\arabic*)},
  font=\normalfont,
  leftmargin=25pt
}
\setlist[itemize]{
  font=\normalfont,
  leftmargin=25pt
}

\newcommand{\kk}{\Bbbk}

\def\rnum#1{\expandafter{\romannumeral #1}}
\def\Rnum#1{\uppercase\expandafter{\romannumeral #1}}

\theoremstyle{plain} 
\newtheorem{thm}{Theorem}[section]
\newtheorem{cor}[thm]{Corollary}
\newtheorem{lem}[thm]{Lemma}
\newtheorem{prop}[thm]{Proposition}

\theoremstyle{definition}
\newtheorem{dfn}[thm]{Definition}
\newtheorem{ex}[thm]{Example}

\newtheorem{rem}[thm]{Remark}

\numberwithin{equation}{section}

\begin{document}

\title
[Point hypergraphs and surface triangulations]
{Hypergraphs associated to skew polynomial algebras \\ and surface triangulations}

\author{Akihiro Higashitani}
\address{Department of Pure and Applied Mathematics, 
Graduate School of Information Science and Technology, 
Osaka University, 
1-5, Yamadaoka, Suita, Osaka 565-0871, Japan}
\email{higashitani@ist.osaka-u.ac.jp}

\author{Kenta Ueyama}
\address{Department of Mathematics,
Faculty of Science,
Shinshu University,
3-1-1 Asahi, Matsumoto, Nagano 390-8621, Japan}
\email{ueyama@shinshu-u.ac.jp}

\subjclass[2020]{Primary 16S38; Secondary 05C65, 05E45}
\keywords{skew polynomial algebras, point schemes, point hypergraphs, 3-uniform hypergraphs, surface triangulations, simplicial complexes}

\begin{abstract}
We study the realization problem for point schemes of skew polynomial algebras.
We approach this problem combinatorially by translating it into the study of certain $3$-uniform hypergraphs, which we call point hypergraphs.
We first give a homological criterion characterizing point hypergraphs.
Using this criterion, we study hypergraphs obtained by deleting triangles from triangulations of connected orientable closed surfaces. Deleting exactly one triangle always gives a non-point hypergraph. When the triangulation has no separating nonfacial $3$-cycles, deleting either no triangles or at least two triangles gives a point hypergraph, and the one-triangle deletions are minimal non-point hypergraphs with respect to taking induced sub-hypergraphs.
It follows that point hypergraphs cannot be characterized by finitely many forbidden induced sub-hypergraphs.
Next, for each point hypergraph, we construct an affine moduli variety of skew polynomial algebras realizing it, and determine its dimension.
For point hypergraphs arising from the above surface construction, we obtain an explicit dimension formula in terms of the number of vertices, the Euler characteristic, and the number of deleted triangles.
Finally, we characterize point hypergraphs on six vertices in terms of a four-vertex local condition together with a single exceptional obstruction.
\end{abstract}

\maketitle

\section{Introduction}

Point schemes are fundamental invariants in noncommutative algebraic geometry, especially in the study of Artin-Schelter regular algebras \cite{Ro}.  Using geometric methods in which point schemes play a central role, Artin, Tate, and Van~den~Bergh studied three-dimensional Artin-Schelter regular algebras and their modules; their work constitutes a foundational achievement in noncommutative algebraic geometry \cite{ATV1, ATV2}.

Beyond the three-dimensional case, the geometry of
point schemes also plays an important role in the study
of four-dimensional AS-regular algebras and their graded modules;
see, for example, \cite{LSV,LS,VV}.
More generally, point schemes and their associated automorphisms are used to study equivalences of graded module categories and noncommutative projective schemes associated to Artin-Schelter regular algebras; see, for example, \cite{CG, Mo, MU}.

Skew polynomial algebras form a basic class of  Artin-Schelter regular algebras.
Their graded module categories and point schemes are studied
using combinatorial methods when the parameters are roots
of unity \cite{HU,HU2}.
Motivated by these results, in this paper, we study the realization problem for point schemes of skew polynomial algebras from a combinatorial point of view.
Throughout this paper, $\kk$ denotes an algebraically closed field of characteristic $0$.  Let
\[
S_a
=\kk \langle x_1,\dots,x_n\rangle/(x_ix_j-a_{ij}x_jx_i \mid 1\le i,j\le n)
\]
be a skew polynomial algebra in $n$ variables, where
$a_{ij}\in \kk^\times$, $a_{ii}=1$, and $a_{ji}=a_{ij}^{-1}$.
By results due to Vit\'oria \cite{Vi} and, independently, to Belmans,
De Laet, and Le Bruyn \cite{BDL}, the point scheme of $S_a$ can be computed explicitly.

\begin{thm}\label{motthm}
The point scheme of $S_a$ is given by
\[
\operatorname{Pts}(S_a)
= \bigcap_{\substack{1\le i<j<k\le n\\ a_{ij}a_{jk}a_{ki}\ne 1}}
\mathcal{V}(x_ix_jx_k)
\quad \subset \mathbb P^{n-1}, 
\]
where $\mathcal{V}(x_ix_jx_k)$ is the zero-locus of $x_ix_jx_k$ in $\mathbb P^{n-1}$. 
\end{thm}

It follows from this theorem that the point scheme of a skew polynomial
algebra in $n$ variables is a union of coordinate linear subspaces in
$\mathbb P^{n-1}$.
This raises a natural realization problem: which unions of coordinate linear subspaces in $\mathbb P^{n-1}$ can occur as the point scheme of a skew polynomial algebra?

To approach this problem, we focus on the coordinate $2$-planes contained in the point scheme.
For a nonempty subset $X\subset [n]:=\{1,2,\dots,n\}$, let
\[
\mathbb P(X)
=\{[p_1:\dots:p_n]\in \mathbb P^{n-1}
\mid p_i=0 \text{ for all } i\notin X\}
\]
be the corresponding coordinate linear subspace.
The theorem above shows that, for distinct $i,j,k$,
\begin{equation}\label{eq.pointdef}
\mathbb P(\{i,j,k\})\subset \operatorname{Pts}(S_a)
\quad\Longleftrightarrow\quad
a_{ij}a_{jk}a_{ki}=1.
\end{equation}
Moreover, the point scheme $\operatorname{Pts}(S_a)$ itself is determined by the coordinate $2$-planes it contains.
Indeed, for any nonempty $X\subset [n]$, we have
\[
\mathbb P(X)\subset \operatorname{Pts}(S_a)
\quad\Longleftrightarrow\quad
\mathbb P(\{i,j,k\})\subset \operatorname{Pts}(S_a)
\text{ for all } \{i,j,k\}\subset X.
\]

This observation suggests translating the realization problem into the language of $3$-uniform hypergraphs.
Namely, we record each coordinate $2$-plane $\mathbb P(\{i,j,k\})$ contained in $\operatorname{Pts}(S_a)$ by the triple
$\{i,j,k\}$.
Then the collection of such triples is naturally regarded as a $3$-uniform hypergraph on $[n]$.
A $3$-uniform hypergraph arising in this way will be called a point hypergraph.
We now make this terminology precise.

\begin{dfn}
Let $V$ be a finite set.  A \textit{$3$-uniform hypergraph} on $V$
is a pair $H=(V,\mathcal H)$ with
$\mathcal H\subset \binom{V}{3}$.  Elements of $\mathcal H$ are called
\textit{faces} of $H$.  Elements of
$\mathcal N:=\binom{V}{3}\setminus \mathcal H$ are called
\textit{non-faces} of $H$.
\end{dfn}

\begin{dfn}\label{def:point}
Let $H=(V,\mathcal H)$ be a $3$-uniform hypergraph.  We say that $H$
is a \textit{point hypergraph} if there exist elements $a_{ij}\in\kk^\times$ for all
ordered pairs of distinct vertices $i,j\in V$ such that the following
conditions hold:
\begin{enumerate}
\item For all distinct $i,j\in V$, one has $a_{ji}=a_{ij}^{-1}$.
\item For every $F\in\binom{V}{3}$ and for one, equivalently for every,
ordering $(i,j,k)$ of its vertices, we have
\[
F\in\mathcal H
\quad\Longleftrightarrow\quad
a_{ij}a_{jk}a_{ki}=1.
\]
\end{enumerate}
If $H$ is not a point hypergraph, we call it \textit{non-point}.
\end{dfn}

We associate to each skew polynomial algebra $S_a$
a $3$-uniform hypergraph
\[
H(S_a)=([n],\mathcal H(S_a))
\]
by setting
\[
\{i,j,k\}\in\mathcal H(S_a)
\quad\Longleftrightarrow\quad
\mathbb P(\{i,j,k\})\subset\operatorname{Pts}(S_a).
\]
By \eqref{eq.pointdef},
$\{i,j,k\}\in\mathcal H(S_a)$
if and only if
$a_{ij}a_{jk}a_{ki}=1$.
Thus $H(S_a)$ is a point hypergraph.

Conversely, if $H=([n],\mathcal H)$ is a point hypergraph on $[n]$, then
a choice of parameters satisfying Definition~\ref{def:point}, together with the conditions $a_{ii}=1$, gives a skew polynomial algebra $S_a$ such that $H=H(S_a)$.
Moreover, the coordinate linear subspace arrangement of $\operatorname{Pts}(S_a)$ is recovered from $H$ as
\[
\operatorname{Pts}(S_a)
=
\bigcup_{\substack{\varnothing \ne X\subset [n]\\ \binom{X}{3}\subset\mathcal H}} \mathbb P(X) \quad \subset \mathbb P^{n-1}.
\]
In this sense, the realization problem for point schemes of skew polynomial algebras is reduced to the combinatorial question of which $3$-uniform hypergraphs are point hypergraphs.

\begin{ex}
Take $n=4$, $a_{34}=-1$, and $a_{ij}=1$ for all other pairs $i<j$.
Then Theorem~\ref{motthm} gives
\[
\operatorname{Pts}(S_a)
=\mathcal{V}(x_1x_3x_4)\cap \mathcal{V}(x_2x_3x_4)
=\mathcal{V}(x_1,x_2)\cup \mathcal{V}(x_3)\cup \mathcal{V}(x_4)
=\mathbb P(\{3,4\})
  \cup\mathbb P(\{1,2,4\})
  \cup\mathbb P(\{1,2,3\}).
\]
Hence the $3$-uniform hypergraph
$H=\bigl([4], \{\{1,2,3\},\{1,2,4\}\}\bigr)$
is a point hypergraph.
\end{ex}

\begin{ex}\label{ex:tetrahedron}
The hypergraph $H=([4], \{\{1,2,4\},\{1,3,4\},\{2,3,4\}\})$ is not a point hypergraph. Indeed, its three faces require
$a_{12}a_{24}a_{41}
=a_{13}a_{34}a_{41}
=a_{23}a_{34}a_{42}=1$.
Since $a_{ji}=a_{ij}^{-1}$, these equalities imply
\[
a_{12}a_{23}a_{31}
=
(a_{12}a_{24}a_{41})(a_{23}a_{34}a_{42})(a_{13}a_{34}a_{41})^{-1}
=1,
\]
which forces $\{1,2,3\}$ to be a face, a contradiction.

Here we regard any four distinct vertices of a point hypergraph as the vertices of a tetrahedron.
The same calculation shows that if three triangular faces of this tetrahedron are faces of the hypergraph, then the fourth must also be a face of the hypergraph.
\end{ex}

Note that, for $n\leq 5$, this realization problem can be solved explicitly by the results of Belmans, De~Laet, and Le~Bruyn~\cite[Section 4]{BDL}.
The present work is motivated by the problem of understanding point hypergraphs more generally.

We now summarize our main results.
The results of this paper are organized around three themes: a homological
criterion for point hypergraphs, local conditions and surface obstructions, and moduli of
realizations for fixed hypergraphs.

First, we give a homological criterion for point hypergraphs by interpreting the face condition $a_{ij}a_{jk}a_{ki}=1$ in terms of the simplicial boundary map.
For a $3$-uniform
hypergraph $H=(V,\mathcal H)$, we define the group
\[G_H
=
C_1(V;\mathbb Z)/
\langle \partial\vv{T}\mid T\in\mathcal H\rangle_{\mathbb Z}.
\]
See Section~\ref{sec:chain} for the precise definition.
We prove that $H$ is a point hypergraph if and only if there exists a group homomorphism
$\rho:G_H\to\kk^\times$ such that $\rho([\partial\vv{N}]_H)\ne 1$
for every non-face $N\in\binom{V}{3}\setminus\mathcal H$
(see Theorem~\ref{thm:char}).
Thus the problem of deciding whether $H$ is a point hypergraph is reduced to a separation problem for the boundary classes of non-faces in the quotient group $G_H$.
This criterion is the main tool used throughout the paper.

Second, we study local conditions for being a point hypergraph.
A basic local necessary condition is the tetrahedron condition:
on every four vertices, three faces force the fourth,
as illustrated in Example~\ref{ex:tetrahedron}.

In the terminology of Belmans, De~Laet, and Le~Bruyn~\cite{BDL}, this condition corresponds exactly to the adequate collection condition.
Their results imply that a $3$-uniform hypergraph on five vertices is a point hypergraph if and only if it satisfies the tetrahedron condition.

Since the five-vertex case is characterized by the tetrahedron condition, it is natural to ask whether such a simple local characterization persists for larger vertex sets.
To address this question, we consider hypergraphs obtained by deleting subsets of triangles from surface triangulations and study when they are point hypergraphs.
Let $K$ be a finite triangulation of a connected orientable closed surface, and, for $D\subset K_2$, set
\[
H_D=(V(K),K_2\setminus D).
\]
We show that $H_D$ is non-point whenever $|D|=1$ (see Theorem~\ref{thm:surface-obst}).
Moreover, we prove that $H_D$ is a point hypergraph for every $D\subset K_2$ with $|D|\ne1$ if and only if $K$ has no separating nonfacial $3$-cycles 
(see Theorem~\ref{thm:minimal}).
Under the same condition, each one-triangle deletion is an induced-minimal
non-point hypergraph. The behavior of the hypergraphs $H_D$ can be summarized as follows.
\[\small
\begin{array}{c|cc}
\toprule
&
\text{$K$ has no separating nonfacial $3$-cycles}
&
\text{$K$ has a separating nonfacial $3$-cycle}
\\
\midrule
|D|=0
& \text{point}
& \text{non-point}
\\
|D|=1
& \text{induced-minimal non-point}
& \text{non-point, not induced-minimal}
\\
|D|\ge2
& \text{point}
& \text{both point and non-point can occur}
\\
\bottomrule
\end{array}
\]
In particular, this construction yields induced-minimal non-point hypergraphs
of arbitrarily large size.
Consequently, point hypergraphs cannot be characterized by finitely many
forbidden induced sub-hypergraphs.

Nevertheless, the six-vertex case still admits a complete description by the tetrahedron condition together with one exceptional obstruction.
We show that a $3$-uniform hypergraph on six vertices is a point hypergraph if and only if it satisfies the tetrahedron condition, except for one isomorphism class: the octahedral obstruction (see Theorem~\ref{thm:tetraocta}).
The proof uses a classification of circuits on six vertices, that is, minimal collections of triangles whose boundaries are linearly dependent over $\mathbb Q$.
Combining this theorem with the count of $175$ isomorphism classes of adequate collections reported in \cite{BDL}, we obtain exactly $174$ point hypergraphs on six vertices up to isomorphism.

Finally, we construct moduli spaces of realizations for fixed hypergraphs.
For a $3$-uniform hypergraph $H$, we construct an affine variety
$\mathcal M_H$ whose points correspond to graded isomorphism classes of skew
polynomial algebras whose associated hypergraph is exactly $H$.
If $H$ is a point hypergraph, then
\[
\dim\mathcal M_H=\operatorname{rank}_{\mathbb Z}G_H
\]
(see Theorem~\ref{thm:moduli-dim}).
For point hypergraphs arising from surface triangulations,
we obtain an explicit dimension formula in terms of the number
of vertices, the Euler characteristic, and the number of deleted
triangles.
Let $K$ be a finite triangulation of a connected orientable closed
surface of Euler characteristic $\chi$. Let $D\subset K_2$, and set $d=|D|$.
If $H_D=(V(K),K_2\setminus D)$ is a point hypergraph,
then $d\ne1$ and
\[
\dim\mathcal M_{H_D}=
\begin{cases}
\displaystyle
\binom{|V(K)|}{2}-2|V(K)|+2\chi+1,
& d=0,\\[2mm]
\displaystyle
\binom{|V(K)|}{2}-2|V(K)|+2\chi+d,
& d\ge2,
\end{cases}
\]
(see Theorem~\ref{thm:surface-moduli}).

\section{Chains, cochains, and the point-hypergraph criterion}\label{sec:chain}

In this section, we translate the definition of a point hypergraph into chain and cochain language.
The idea is to view the assignment $(i,j)\mapsto a_{ij}$ as a cochain, so that the product $a_{ij}a_{jk}a_{ki}$ is its value on the boundary of the oriented triangle
$(i,j,k)$.
We use the standard notions of simplicial chains and cochains,
which we recall explicitly below for convenience; see \cite[Sections~2.1 and~3.1]{Ha}.

Let $V$ be a finite set.
For notational convenience, we often write $ijk$ for
the three-element subset $\{i,j,k\}\subset V$.
We denote by $C_1(V;\mathbb Z)$ the free abelian group generated by ordered pairs $(i,j)$ with $i,j\in V$ and $i\ne j$, subject to the relations $(j,i)=-(i,j)$.
Thus the class of $(i,j)$ in $C_1(V;\mathbb Z)$ represents an oriented edge from $i$ to $j$.
Similarly, $C_2(V;\mathbb Z)$ denotes the free abelian group generated by ordered triples $(i,j,k)$ of distinct elements of $V$, subject to the alternating relations $(i_{\tau(1)},i_{\tau(2)},i_{\tau(3)}) = \operatorname{sgn}(\tau)(i_1,i_2,i_3)$ for all $\tau\in \mathfrak{S}_3$.
Thus the class of $(i,j,k)$ in $C_2(V;\mathbb Z)$ represents an oriented triangle. The boundary map $\partial:C_2(V;\mathbb Z)\to C_1(V;\mathbb Z)$ is defined by
\[\partial(i,j,k)=(i,j)+(j,k)+(k,i) \]
and extended linearly.
Since $(k,i)=-(i,k)$, this is the usual simplicial boundary of an oriented triangle.

We shall use cochains with values in the multiplicative group $\kk^\times$.
Thus
\[
C^1(V;\kk^\times):=\operatorname{Hom}(C_1(V;\mathbb Z),\kk^\times),
\qquad
C^2(V;\kk^\times):=\operatorname{Hom}(C_2(V;\mathbb Z),\kk^\times).
\]
For $\varphi\in C^1(V;\kk^\times)$, define $\delta\varphi\in C^2(V;\kk^\times)$ by
\[
(\delta\varphi)(i,j,k)
=\varphi(i,j)\varphi(j,k)\varphi(k,i).
\]
Then, for every $z\in C_2(V;\mathbb Z)$, one has
\begin{equation}\label{eq:dd}
(\delta\varphi)(z)=\varphi(\partial z).
\end{equation}

Let $H=(V,\mathcal H)$ be a $3$-uniform hypergraph.
For each $F\in\binom{V}{3}$, choose an orientation of $F$, and denote the corresponding element of $C_2(V;\mathbb Z)$ by $\vv{F}$.
If $\varphi\in C^1(V;\kk^\times)$, then replacing $\vv{F}$ by $-\vv{F}$ replaces
$(\delta\varphi)(\vv{F})$ by its inverse.
Thus whether $(\delta\varphi)(\vv{F})$ is equal to $1$ does not depend on the orientation chosen for $F$.

The following lemma rewrites Definition~\ref{def:point} in cochain language.

\begin{lem}\label{lem:cochain}
Let $H=(V,\mathcal H)$ be a $3$-uniform hypergraph. Then $H$ is a point hypergraph if and only if there exists $\varphi\in C^1(V;\kk^\times)$ such that
$F\in\mathcal H$ if and only if $(\delta\varphi)(\vv{F})=1$
for every $F\in\binom{V}{3}$.
\end{lem}

\begin{proof}
Suppose first that $H$ is a point hypergraph, and choose labels $a_{ij}$ as in Definition~\ref{def:point}.
They define a cochain $\varphi_a\in C^1(V;\kk^\times)$ by $\varphi_a(i,j)=a_{ij}$, since $a_{ji}=a_{ij}^{-1}$.
If $\vv{F}=(i,j,k)$, then
\[ (\delta\varphi_a)(\vv{F})=a_{ij}a_{jk}a_{ki}, \]
so the desired condition follows.

Conversely, suppose that such a cochain $\varphi$ exists. Set $a_{ij}:=\varphi(i,j)$ for all ordered pairs $i\ne j$.
Then $a_{ji}=\varphi(j,i)=\varphi(-(i,j))=\varphi(i,j)^{-1}=a_{ij}^{-1}$.
Moreover, if $\vv{F}=(i,j,k)$, then
\[ a_{ij}a_{jk}a_{ki} = \varphi(i,j)\varphi(j,k)\varphi(k,i) = (\delta\varphi)(\vv{F}). \]
Thus the labels $a_{ij}$ satisfy the conditions in Definition~\ref{def:point}, and therefore $H$ is a point hypergraph.
\end{proof}

The above lemma gives a criterion for being a point hypergraph in terms of cochains.
We now reformulate this criterion in terms of a quotient of $C_1(V;\mathbb Z)$.

Let $H=(V,\mathcal H)$ be a $3$-uniform hypergraph.  Choose an orientation
$\vv{T}$ for each face $T\in\mathcal H$, and define
\[
L_H:=
\langle \partial \vv{T} \mid T\in\mathcal H \rangle_{\mathbb Z} \subset C_1(V;\mathbb Z).
\]
This subgroup does not depend on the choices of orientations, since replacing
$\vv{T}$ by $-\vv{T}$ only changes the corresponding generator by a sign.  Set
\[
G_H:=C_1(V;\mathbb Z)/L_H,
\]
and denote by $[x]_H$ the image of $x\in C_1(V;\mathbb Z)$ in $G_H$.

For a non-face $N\in\binom{V}{3}\setminus\mathcal H$, choose an orientation
$\vv{N}$.
The class $[\partial\vv{N}]_H$ changes sign if the orientation of
$N$ is reversed.  Hence, for any homomorphism $\rho:G_H\to\kk^\times$, the
condition
$\rho([\partial\vv{N}]_H)\ne 1$
is independent of the chosen orientation of $N$.

\begin{thm}\label{thm:char}
Let $H=(V,\mathcal H)$ be a $3$-uniform hypergraph.
Then $H$ is a point hypergraph if and only if there exists a group homomorphism $\rho:G_H\to\kk^\times$ such that
\[
\rho([\partial\vv{N}]_H)\ne 1
\]
for every non-face $N\in\binom{V}{3}\setminus\mathcal H$.
\end{thm}

\begin{proof}
Suppose first that $H$ is a point hypergraph.
By Lemma~\ref{lem:cochain} and \eqref{eq:dd},
there exists $\varphi\in C^1(V;\kk^\times)$ such that
$F\in\mathcal H$ if and only if
$\varphi(\partial\vv{F})=1$
for every $F\in\binom{V}{3}$.
In particular, $\varphi$ is trivial on $L_H$,
and therefore factors through the quotient $G_H$.
Let $\rho:G_H\to\kk^\times$ be the induced homomorphism.
If $N\notin\mathcal H$,
then $\rho([\partial\vv{N}]_H)=\varphi(\partial\vv{N})\ne 1$.
Hence $\rho$ has the required property.

Conversely, suppose that such a homomorphism $\rho:G_H\to\kk^\times$ exists.
Let $\pi:C_1(V;\mathbb Z)\to G_H$ be the quotient map, and set
\[
\varphi:=\rho\circ\pi\in C^1(V;\kk^\times).
\]
If $T\in\mathcal H$, then $\partial\vv{T}\in L_H$, so
$\varphi(\partial\vv{T})=\rho(\pi(\partial\vv{T}))=1$.
If $N\notin\mathcal H$, then by assumption
$\varphi(\partial\vv{N})=\rho([\partial\vv{N}]_H)\ne 1$.
Thus $F\in\mathcal H$ if and only if
$\varphi(\partial\vv{F})=1$
for every $F\in\binom{V}{3}$.
By Lemma~\ref{lem:cochain} and \eqref{eq:dd}, $H$ is a point hypergraph.
\end{proof}

\begin{ex}
We reinterpret Example~\ref{ex:tetrahedron} using Theorem~\ref{thm:char}.
Let $H$ be the hypergraph in that example.
Since $\partial(1,2,3)
=\partial(1,2,4)-\partial(1,3,4)+\partial(2,3,4)
\in L_H$,
we have $[\partial(1,2,3)]_H=0$.
Since $\{1,2,3\}$ is a non-face, no homomorphism $\rho:G_H\to\kk^\times$ can satisfy the condition in Theorem~\ref{thm:char}. Hence $H$ is non-point.
\end{ex}

\section{Hypergraphs from triangulations of orientable closed surfaces}

In this section, we study $3$-uniform hypergraphs associated to finite triangulations of orientable closed surfaces. 

For a finite triangulation $K$ of a closed surface, we denote by $V(K)$ its vertex set and by $K_2$ the set of its triangles.
The following theorem generalizes Example~\ref{ex:tetrahedron}
from the boundary of a tetrahedron to triangulations of connected orientable closed surfaces.

\begin{thm}\label{thm:surface-obst}
Let $K$ be a finite triangulation of a connected orientable closed surface.
For any triangle $F^*\in K_2$, the $3$-uniform hypergraph
\[
H=(V(K), K_2\setminus\{F^*\})
\]
is a non-point hypergraph.
\end{thm}

\begin{proof}
Fix an orientation of the surface, and give each triangle of $K$ the induced orientation. Then $Z :=\sum_{T\in K_2}\vv{T}$ is a $2$-cycle, i.e., $\partial Z=0$. Therefore, 
\[
\partial\vv{F^*}
=-\sum_{T\in K_2\setminus\{F^*\}}\partial\vv{T}.
\]
Every triangle $T\in K_2\setminus\{F^*\}$ is a face of $H$, so the right-hand
side belongs to $L_H$.
Thus
$[\partial\vv{F^*}]_H=0$
in $G_H$.
Hence every homomorphism $\rho:G_H\to\kk^\times$ satisfies
$\rho([\partial\vv{F^*}]_H)=1$.
Since $F^*$ is a non-face of $H$, Theorem~\ref{thm:char} implies that $H$ is non-point.
\end{proof}

\begin{cor}\label{cor:induced-obst}
Let $H$ be a $3$-uniform hypergraph.
Suppose that $H$ has an induced sub-hypergraph isomorphic to
\[
(V(K),K_2\setminus\{F^*\})
\]
for some finite triangulation $K$ of a connected orientable closed surface and some $F^*\in K_2$.
Then $H$ is a non-point hypergraph.
\end{cor}

\begin{proof}
By Theorem~\ref{thm:surface-obst}, the induced sub-hypergraph $(V(K),K_2\setminus\{F^*\})$ is non-point.
Since induced sub-hypergraphs of point hypergraphs are again point hypergraphs,
$H$ is non-point.
\end{proof}

If a non-point hypergraph contains a proper induced sub-hypergraph which is already a non-point hypergraph, then the obstruction is detected on fewer vertices.
We therefore focus on the following minimal obstructions.

\begin{dfn}\label{def:induced-minimal-non-point}
A non-point hypergraph is called \textit{induced-minimal non-point} if all of its proper induced sub-hypergraphs are point hypergraphs.
\end{dfn}

\begin{ex}
Let $H$ be the hypergraph obtained from the boundary of a tetrahedron by deleting one face.
Then $H$ is non-point by Theorem~\ref{thm:surface-obst}.
On the other hand, every proper induced sub-hypergraph of $H$ has at most
three vertices, and hence is a point hypergraph.
Therefore $H$ is induced-minimal non-point.
\end{ex}

\begin{ex}
Let $H$ be the hypergraph obtained from the boundary of a triangular bipyramid with vertex set $\{u,v,a,b,c\}$ by deleting the face $uab$.
Then $H$ is a non-point hypergraph by Theorem~\ref{thm:surface-obst}.
We check that $H$ is not induced-minimal.
Consider the induced sub-hypergraph $H|_{\{v,a,b,c\}}$.
It is obtained from the boundary of a tetrahedron by deleting the face $abc$, as illustrated below.
\[
\begin{tikzpicture}[
  scale=0.95,
  face/.style={draw=black, fill=gray!18, line join=round},
  missing/.style={draw=black, dashed, fill=white, line width=0.8pt, line join=round},
  every node/.style={font=\scriptsize}
]
  \newcommand{\uptri}[4]{\path[#1] (#2,#3) -- (#2+1,#3) -- (#2+0.5,#3+0.86) -- cycle;
    \node at (#2+0.5,#3+0.30) {$#4$};
  }
  \newcommand{\downtri}[4]{\path[#1] (#2,#3+0.86) -- (#2+1,#3+0.86) -- (#2+0.5,#3) -- cycle;
    \node at (#2+0.5,#3+0.54) {$#4$};
  }
  \newcommand{\cycletri}[3]{\draw[dashed, line width=0.8pt]
      (#1,#2+0.7) -- (#1+1.15,#2+0.7) -- (#1+0.575,#2) -- cycle;
    \node at (#1+0.575,#2+0.45) {$#3$};
  }

  \uptri{missing}{0}{2.0}{uab}
  \uptri{face}{1.25}{2.0}{uac}
  \uptri{face}{2.50}{2.0}{ubc}

  \cycletri{1.17}{0.95}{abc}

  \downtri{face}{0}{-0.25}{vab}
  \downtri{face}{1.25}{-0.25}{vac}
  \downtri{face}{2.50}{-0.25}{vbc}

  \node[font=\Large] at (4.8,1.2) {$\supsetneq$};

  \cycletri{6.92}{0.95}{abc}

  \downtri{face}{5.75}{-0.25}{vab}
  \downtri{face}{7.0}{-0.25}{vac}
  \downtri{face}{8.25}{-0.25}{vbc}

  \node[font=\footnotesize, align=center] at (1.85,-0.75)
    {triangular bipyramid minus $uab$};

  \node[font=\footnotesize, align=center] at (7.5,-0.75)
    {tetrahedron minus $abc$};
\end{tikzpicture}
\]
By Theorem~\ref{thm:surface-obst}, $H|_{\{v,a,b,c\}}$ is also non-point.
Hence $H$ is non-point but not induced-minimal.
\end{ex}

To check induced minimality, we need a criterion for
showing that proper induced sub-hypergraphs are point hypergraphs.
For this purpose, we first give a separation lemma.

\begin{lem}\label{lem:separation} 
Let $G$ be a finitely generated abelian group, and let $S\subset G$ be a finite set.
Suppose that $s\otimes 1\ne0$ in $G\otimes_{\mathbb Z}\mathbb Q$ for every $s\in S$.
Then there exists a group homomorphism $\rho:G\to\kk^\times$ such that $\rho(s)\ne1$ for every $s\in S$.
\end{lem}

\begin{proof}
Set $V=G\otimes_{\mathbb Z}\mathbb Q$ and
$W=\operatorname{Hom}_{\mathbb Q}(V,\mathbb Q)$.
For each $s\in S$, set
$W_s=\{\psi\in W\mid \psi(s\otimes1)=0\}$.
Since $s\otimes1\ne0$, each $W_s$ is a proper subspace of $W$.
As $S$ is finite and $\mathbb Q$ is infinite, we may choose
$\psi\in W\setminus\bigcup_{s\in S}W_s$.
Then $\psi(s\otimes1)\ne0$ for every $s\in S$.

Since $G$ is finitely generated, after multiplying $\psi$ by a positive
integer we may assume that $\psi(g\otimes1)\in\mathbb Z$ for every $g\in G$.
Define $\nu:G\to\mathbb Z$ by $\nu(g)=\psi(g\otimes1)$.
Then $\nu$ is a group homomorphism and $\nu(s)\ne0$ for every $s\in S$.
Define a group homomorphism $\rho:G\to\kk^\times$ by $\rho(g)=2^{\nu(g)}$.
Since $\kk$ has characteristic $0$, we have $2^m\ne1$ for every $m\ne0$.
Hence $\rho(s)\ne1$ for every $s\in S$.
\end{proof}

The following consequence of Lemma~\ref{lem:separation} will be used below.

\begin{lem}\label{lem:rational-point-criterion}
Let $H=(V,\mathcal H)$ be a $3$-uniform hypergraph.
Suppose that
\[
[\partial\vv{N}]_H\otimes1\ne0
\quad\text{in }G_H\otimes_{\mathbb Z}\mathbb Q
\]
for every non-face $N$ of $H$.
Then $H$ is a point hypergraph.
\end{lem}

\begin{proof}
Set $S:=\{[\partial\vv{N}]_H\mid N \text{ is a non-face of } H\}\subset G_H$.
By assumption, every element of $S$ is nonzero after tensoring with $\mathbb Q$. 
Hence Lemma~\ref{lem:separation} gives a group homomorphism
$\rho:G_H\to\kk^\times$ such that $\rho([\partial\vv{N}]_H)\ne1$
for every non-face $N$ of $H$.  Therefore $H$ is a point hypergraph by
Theorem~\ref{thm:char}.
\end{proof}

The next lemma gives a different sufficient condition for being a point hypergraph.

\begin{lem}\label{lem:delete-criterion}
Let $H=(V,\mathcal H)$ be a $3$-uniform hypergraph, and let
$F\in\binom{V}{3}\setminus\mathcal H$.
Set $H^+=(V,\mathcal H\cup\{F\})$.
Suppose that $H^+$ is a point hypergraph and
\[
[\partial\vv{F}]_H\otimes1\ne0
\quad\text{in }G_H\otimes_{\mathbb Z}\mathbb Q.
\]
Then $H$ is a point hypergraph.
\end{lem}

\begin{proof}
By Theorem~\ref{thm:char}, there exists a group homomorphism
$\rho_0:G_{H^+}\to\kk^\times$
such that
$\rho_0([\partial\vv{N}]_{H^+})\ne1$
for every non-face $N$ of $H^+$.

Let $\pi:G_H\to G_{H^+}$ be the natural quotient map.
By the proof of Lemma~\ref{lem:separation}, the assumption $[\partial\vv{F}]_H\otimes1\ne0$ implies the existence of a group homomorphism $\nu:G_H\to\mathbb Z$
such that $\nu([\partial\vv{F}]_H)\ne0$.
For $t\in\kk^\times$, we define a group homomorphism
\[
\rho_t:G_H\to\kk^\times; \ g \mapsto \rho_0(\pi(g))t^{\nu(g)}.
\]

For the new non-face $F$, since $\pi([\partial\vv{F}]_H)=0$, we have
$\rho_t([\partial\vv{F}]_H)=t^{\nu([\partial\vv{F}]_H)}$.
Since $\nu([\partial\vv{F}]_H)\ne0$, the condition
$\rho_t([\partial\vv{F}]_H)=1$
can hold for only finitely many values of $t\in\kk^\times$.

Let $N$ be a non-face of $H^+$.
If $\nu([\partial\vv{N}]_H)=0$,
then $\rho_t([\partial\vv{N}]_H)
= \rho_0([\partial\vv{N}]_{H^+})
\ne1$ for every $t \in \kk^\times$.
If $\nu([\partial\vv{N}]_H)\ne0$,
then the condition
$\rho_t([\partial\vv{N}]_H)=1$
can hold for only finitely many values of $t\in\kk^\times$.

Since there are only finitely many non-faces and $\kk^\times$ is infinite,
we may choose $t\in\kk^\times$ avoiding all exceptional values.
For this choice of $t$, we have
$\rho_t([\partial\vv{N}]_H)\ne1$
for every non-face $N$ of $H$.
Hence $H$ is a point hypergraph by Theorem~\ref{thm:char}.
\end{proof}

\begin{dfn}
Let $K$ be a finite triangulation of a closed surface.  A \textit{$3$-cycle} in $K$ is a three-element subset $C=\{i,j,k\}\subset V(K)$ 
such that every two-element subset of $C$ is an edge of $K$.
A $3$-cycle $C$ is called \textit{facial} if $C\in K_2$, and \textit{nonfacial} otherwise.
If $C$ is nonfacial, then the three edges spanned by $C$ form a simple closed curve $\gamma_C\subset |K|$, where $|K|$ is a geometric realization of $K$.
We say that $C$ is \textit{separating} if the curve $\gamma_C$ is separating, that is, if $|K|\setminus\gamma_C$ is disconnected.
\end{dfn}

\begin{rem}\label{rem:flag}
We recall that a simplicial complex is \textit{flag} if every minimal non-face has cardinality two.
Let $K$ be a finite triangulation of a closed surface.  If $K$ is flag, then $K$ has no nonfacial $3$-cycles.
Indeed, a nonfacial $3$-cycle is a minimal non-face of cardinality three, so no such cycle can occur in a flag complex.
\end{rem}

The following is a standard fact from surface topology; see, for example, \cite[Section 1.3]{FM}.

\begin{lem}\label{lem:sepcurve}
Let $\Sigma$ be a connected orientable closed surface, and let
$\gamma\subset\Sigma$ be a simple closed curve.  Then $\gamma$ is separating
if and only if its homology class $[\gamma]\in H_1(\Sigma;\mathbb Z)$ is zero.
\end{lem}

We now characterize when the one-triangle obstructions arising from finite
triangulations of orientable closed surfaces are induced-minimal non-point.
The same condition is also equivalent to the property that every hypergraph obtained by deleting either no triangles or at least two triangles is a point hypergraph.
For a subset $D\subset K_2$, we write
\[
H_D=(V(K),K_2\setminus D).
\]

\begin{thm}\label{thm:minimal}
Let $K$ be a finite triangulation of a connected orientable closed surface, and let $F^*\in K_2$.
Set
\begin{align*}
H=(V(K),K_2\setminus\{F^*\})
\qquad \text{and} \qquad
H^{+}=(V(K),K_2).
\end{align*}
Thus $H=H_{\{F^*\}}$ and $H^+=H_\varnothing$.
In view of Theorem~\ref{thm:surface-obst}, $H$ is a non-point hypergraph.
The following conditions are equivalent.
\begin{enumerate}
\item $H$ is an induced-minimal non-point hypergraph.
\item $H^{+}$ is a point hypergraph.
\item $H_D$ is a point hypergraph for every $D\subset K_2$ with $|D|\ne1$.
\item The triangulation $K$ has no separating nonfacial $3$-cycles.
\end{enumerate}
\end{thm}

\begin{proof}
We prove (1) $\Rightarrow$ (4) $\Rightarrow$ (3) $\Rightarrow$ (2) $\Rightarrow$ (1).

(1) $\Rightarrow$ (4) :
We prove the contrapositive.
Suppose that $K$ has a separating nonfacial $3$-cycle $C$.
Let $\Omega$ be a component of $|K|\setminus\gamma_C$ which does not contain
the interior of $F^*$, and set
\[
\mathcal T=\{T\in K_2\mid \text{the interior of $T$ is contained in $\Omega$}\}.
\]
Let $U$ be the set of vertices contained in the triangles of $\mathcal T$.
Then $U\subsetneq V(K)$ and $C\in\binom{U}{3}$.
Orienting the triangles in $\mathcal T$ compatibly with the orientation of
$K$, we have
\[
\partial\left(\sum_{T\in\mathcal T}\vv{T}\right)
=
\pm\partial\vv{C}.
\]
Since $F^*\notin\mathcal T$, every triangle in $\mathcal T$ is a face of
$H|_U$.  Hence $[\partial\vv{C}]_{H|_U}=0$ in $G_{H|_U}$.
Since $C$ is a non-face of $H|_U$, Theorem~\ref{thm:char} implies that $H|_U$ is non-point.
Therefore $H$ is non-point but not induced-minimal.

(4) $\Rightarrow$ (3):
Suppose that $K$ has no separating nonfacial $3$-cycles.
Let $D\subset K_2$ with $|D|\ne1$, and let $N$ be a non-face of $H_D$.
We show that $[\partial\vv{N}]_{H_D}\otimes1\ne0$
in $G_{H_D}\otimes_{\mathbb Z}\mathbb Q$.

First suppose that $N\notin K_2$.
If some edge $e$ of $N$ is not an edge of $K$, then the $e$-coefficient of
$\partial\vv{N}$ is $\pm1$, whereas every element of
$\langle \partial\vv{T}\mid T\in K_2\setminus D\rangle_{\mathbb Q}$
has $e$-coefficient $0$.
Hence the claim holds in this case.
It remains to consider the case where all three edges of $N$ are edges of
$K$.
Then $N$ is a nonfacial $3$-cycle of $K$.
Suppose, to the contrary, that
$[\partial\vv{N}]_{H_D}\otimes 1=0$.
Then $\partial\vv{N}$ is the boundary of a rational $2$-chain on $K$.
Hence $[\gamma_N] \in H_1(|K|;\mathbb Q)$ is zero.
Since $|K|$ is an orientable closed surface, $H_1(|K|;\mathbb Z)$ is torsion-free, 
so $[\gamma_N] \in H_1(|K|;\mathbb Z)$ is also zero.
By Lemma~\ref{lem:sepcurve}, $\gamma_N$ is separating, contrary to the
assumption.
Hence $[\partial\vv{N}]_{H_D}\otimes1\ne0$.

Next suppose that $N\in D$.
Since $|D|\ne1$, we have $|D|\ge2$.
Suppose, to the contrary, that
$[\partial\vv{N}]_{H_D}\otimes1=0$.
Then there exists a rational $2$-chain $c$ supported on the faces of $H_D$
such that $\partial c=\partial\vv{N}$.
Hence $z:=c-\vv{N}$ is a nonzero rational $2$-cycle on $K$.
Write $z=\sum_{T\in K_2}a_T\vv{T}$.
If two triangles $T$ and $T'$ share an edge, then comparing the coefficient
of that edge in $\partial z=0$ gives $a_T=\pm a_{T'}$.
Hence the support of $z$ is closed under passing to adjacent triangles.
Since the dual graph of the connected closed surface $K$ is connected and
$z\ne0$, every triangle of $K$ belongs to the support of $z$.
On the other hand, choose $F\in D\setminus\{N\}$.
Then $F$ appears neither in $c$ nor in $\vv{N}$, and hence does not belong
to the support of $z$, a contradiction.
Therefore $[\partial\vv{N}]_{H_D}\otimes1\ne0$.

Thus the desired nonvanishing holds for every non-face of $H_D$.
By Lemma~\ref{lem:rational-point-criterion}, $H_D$ is a point hypergraph.
This proves (3).

(3) $\Rightarrow$ (2):
Taking $D=\varnothing$ in (3), we obtain that $H^+=H_\varnothing$ is a point hypergraph.

(2) $\Rightarrow$ (1):
Suppose that $H^+$ is a point hypergraph.
As noted above, $H$ is non-point by Theorem~\ref{thm:surface-obst}.
It remains to prove that every proper induced sub-hypergraph of $H$ is a point hypergraph.

Let $U\subsetneq V(K)$. If $F^*\not\subset U$, then $H|_U=H^+|_U$.
Since induced sub-hypergraphs of point hypergraphs are again point hypergraphs,
$H|_U$ is a point hypergraph.

Assume now that $F^*\subset U$.
Since $H^+$ is a point hypergraph, the induced sub-hypergraph $H^+|_U$ is a point hypergraph.
We show that $[\partial\vv{F^*}]_{H|_U}\otimes1\ne0$
in $G_{H|_U}\otimes_{\mathbb Z}\mathbb Q$.
Suppose otherwise.
Then there exists a rational $2$-chain $c$ supported on
the faces of $H|_U$ such that $\partial c=\partial\vv{F^*}$.
Hence
$z:=c-\vv{F^*}$ is a nonzero rational $2$-cycle on $K$, since $F^*$ does not appear in $c$.
Write $z=\sum_{T\in K_2} a_T\vv{T}$.
If two triangles $T$ and $T'$ share an edge, then comparing the coefficient of that edge in $\partial z=0$ gives $a_T=\pm a_{T'}$.
Hence the support of $z$ is closed under passing to adjacent triangles.
Since the dual graph of the connected closed surface $K$ is connected and $z\ne0$, every triangle of $K$ belongs to the support of $z$.
On the other hand, choose a vertex $v\in V(K)\setminus U$.
Since $K$ is a triangulation of a closed surface, there is a triangle $T\in K_2$ containing
$v$.  Then $T\notin\binom{U}{3}$, so $T$ does not appear in $c$.
Since $F^*\subset U$, we see that $T \ne F^*$.
Hence $T$ does not appear in $z=c-\vv{F^*}$, a contradiction.
Therefore $[\partial\vv{F^*}]_{H|_U}\otimes1\ne0$.
By Lemma~\ref{lem:delete-criterion}, applied to
$H|_U\subset H^+|_U$, the hypergraph $H|_U$ is a point hypergraph.

Since $H$ is non-point and $H|_U$ is a point hypergraph for every proper subset $U\subsetneq V(K)$, it follows that $H$ is induced-minimal non-point.
\end{proof}

\begin{cor}\label{cor:flag-minimal}
Let $K$ be a finite flag triangulation of a connected orientable closed surface, and let $F^*\in K_2$. Then
$H=(V(K),K_2\setminus\{F^*\})$
is an induced-minimal non-point hypergraph and 
$H^{+}=(V(K),K_2)$
is a point hypergraph.
\end{cor}

\begin{proof}
By Remark~\ref{rem:flag}, $K$ has no nonfacial $3$-cycles.
Hence the assertion follows from Theorem~\ref{thm:minimal}.
\end{proof}

\begin{ex}
Let $K$ be a finite triangulation of the $2$-sphere $S^2$, and let $F^*\in K_2$.
Since every simple closed curve on $S^2$ is separating, Theorem~\ref{thm:minimal} shows that $(V(K),K_2\setminus\{F^*\})$ is induced-minimal non-point if and only if $K$ has no nonfacial $3$-cycles.
This condition is equivalent to saying that
$K$ is flag or the boundary complex of a tetrahedron.
\end{ex}

\begin{ex}
Let $T_{3,3}$ be the standard $3\times3$ triangulation of the torus, shown
below in a fundamental domain.
\begin{center}
\begin{tikzpicture}[
  scale=1.05,
  v/.style={circle, fill, inner sep=1.1pt},
  hv/.style={circle, fill, inner sep=1.8pt},
  lbl/.style={fill=white, inner sep=0.8pt, font=\scriptsize},
  every node/.style={font=\scriptsize}
]
  \foreach \i in {0,1,2,3}{
    \draw[gray!50] (\i,0) -- (\i,3);
    \draw[gray!50] (0,\i) -- (3,\i);
  }
  \foreach \i in {0,1,2}{
    \foreach \j in {0,1,2}{
      \draw[gray!50] (\i+1,\j) -- (\i,\j+1);
    }
  }

  \draw[gray!70] (0,0) rectangle (3,3);

  \draw[gray!70, line width=0.5pt] (-0.10,1.45) -- (0,1.6) -- (0.10,1.45);
  \draw[gray!70, line width=0.5pt] (2.90,1.45) -- (3,1.6) -- (3.10,1.45);

  \draw[gray!70, line width=0.5pt] (1.35,-0.10) -- (1.5,0) -- (1.35,0.10);
  \draw[gray!70, line width=0.5pt] (1.55,-0.10) -- (1.7,0) -- (1.55,0.10);

  \draw[gray!70, line width=0.5pt] (1.35,2.90) -- (1.5,3) -- (1.35,3.10);
  \draw[gray!70, line width=0.5pt] (1.55,2.90) -- (1.7,3) -- (1.55,3.10);

  \draw[very thick, line cap=round] (0,1) -- (3,1);

  \foreach \i in {0,1,2,3}{
    \foreach \j in {0,1,2,3}{
      \node[v] at (\i,\j) {};
    }
  }

  \foreach \p in {(0,1),(1,1),(2,1),(3,1)}{
    \node[v] at \p {};
  }

  \node[lbl, anchor=south west] at (0.04,0.04) {$1$};
  \node[lbl, anchor=south west] at (1.04,0.04) {$2$};
  \node[lbl, anchor=south west] at (2.04,0.04) {$3$};
  \node[lbl, anchor=south west] at (3.04,0.04) {$1$};

  \node[lbl, anchor=south west] at (0.04,1.04) {$7$};
  \node[lbl, anchor=south west] at (1.04,1.04) {$8$};
  \node[lbl, anchor=south west] at (2.04,1.04) {$9$};
  \node[lbl, anchor=south west] at (3.04,1.04) {$7$};

  \node[lbl, anchor=south west] at (0.04,2.04) {$4$};
  \node[lbl, anchor=south west] at (1.04,2.04) {$5$};
  \node[lbl, anchor=south west] at (2.04,2.04) {$6$};
  \node[lbl, anchor=south west] at (3.04,2.04) {$4$};

  \node[lbl, anchor=south west] at (0.04,3.04) {$1$};
  \node[lbl, anchor=south west] at (1.04,3.04) {$2$};
  \node[lbl, anchor=south west] at (2.04,3.04) {$3$};
  \node[lbl, anchor=south west] at (3.04,3.04) {$1$};
\end{tikzpicture}
\end{center}
It is not flag.
For example, the thick curve represents the nonfacial $3$-cycle $C=\{7,8,9\}$.
On the one hand, the nonfacial $3$-cycles of $T_{3,3}$
are essential curves on the torus, and hence none of them is separating.
Therefore, by Theorem~\ref{thm:minimal}, for any triangle
$F^*\in (T_{3,3})_2$, the hypergraph
$(V(T_{3,3}),(T_{3,3})_2\setminus\{F^*\})$
is induced-minimal non-point.

On the other hand, $(V(T_{3,3}),(T_{3,3})_2)$
is a point hypergraph. Indeed, 
\[
(T_{3,3})_2=
\{
145, 125, 256, 236, 346, 134,
478, 458, 589, 569, 679, 467,
127, 278, 238, 389, 139, 179
\}
\]
is realized by the following parameters:
\[
\begin{array}{c|ccccccccc}
a_{ij} & 1&2&3&4&5&6&7&8&9\\
\hline\\[-1em]
1&1&2^{-1}&2^{-1}&2^{-1}&2^{-1}&-1&2^{-1}&-1&1\\
2&2&1&2^{-1}&-1&1&2^{-1}&1&1&-1\\
3&2&2&1&1&-1&1&-1&2&2\\
4&2&-1&1&1&1&1&2^{-1}&2^{-1}&-1\\
5&2&1&-1&1&1&2^{-1}&-1&2^{-1}&2^{-1}\\
6&-1&2&1&1&2&1&2^{-1}&-1&1\\
7&2&1&-1&2&-1&2&1&1&2\\
8&-1&1&2^{-1}&2&2&-1&1&1&1\\
9&1&-1&2^{-1}&-1&2&1&2^{-1}&1&1
\end{array}
\]
\end{ex}

We now show that induced-minimal non-point hypergraphs can have arbitrarily large size.
Thus the obstructions to being a point hypergraph are not bounded in size.

\begin{cor}\label{cor:large}
For every integer $\ell$, there exists an induced-minimal non-point hypergraph
with more than $\ell$ vertices.
\end{cor}

\begin{proof}
For each $n\ge4$, let $K_n$ be the $n$-gonal bipyramid, with vertices $u,v,0,1,\ldots,n-1$ and triangles
$\{u,i,i+1\}$ and $\{v,i,i+1\}$ for $i\in\mathbb Z/n\mathbb Z$.
Then $K_n$ is a flag triangulation of $S^2$.
Indeed, since the base cycle has no diagonals and $u$ and $v$ are not joined by an edge, every minimal non-face has cardinality two.

Choose $n$ with $n+2>\ell$.
For any triangle $F^*\in (K_n)_2$, the hypergraph
\[
H_n=(V(K_n),(K_n)_2\setminus\{F^*\})
\]
is induced-minimal non-point by Corollary~\ref{cor:flag-minimal}.
Since $|V(K_n)|=n+2>\ell$, this gives the desired example.
\end{proof}

This shows that being a point hypergraph cannot be characterized by a finite list of forbidden induced sub-hypergraphs.

\begin{cor}\label{cor:forbidden}
There is no finite family $\mathcal F$ of $3$-uniform hypergraphs such that a $3$-uniform hypergraph is a point hypergraph if and only if it contains no member of $\mathcal F$ as an induced sub-hypergraph.
\end{cor}

\begin{proof}
Suppose that such a finite family $\mathcal F$ exists.
Let $H$ be an induced-minimal non-point hypergraph.
Then $H$ contains an induced sub-hypergraph isomorphic to some
$F\in\mathcal F$.
Since $F$ is non-point and $H$ is induced-minimal non-point, this induced sub-hypergraph cannot be proper.
Hence $H\cong F$.
Thus the sizes of induced-minimal non-point hypergraphs are bounded by $\max\{|V(F)|\mid F\in\mathcal F\}$, contradicting
Corollary~\ref{cor:large}.
\end{proof}

\begin{rem}\label{rem:both-deletions}
When $K$ has a separating nonfacial $3$-cycle, both point and non-point
hypergraphs can occur among $H_D$ with $|D|\ge2$.
For example, let $K$ be the boundary of the triangular bipyramid with vertex set
$\{u,v,a,b,c\}$, where $u$ and $v$ are the two apices.
Then $\{a,b,c\}$ is a separating nonfacial $3$-cycle.
For $D=\{\{u,a,b\},\{u,a,c\}\}$,
the hypergraph $H_D$ is non-point, whereas for
$D=\{\{u,a,b\},\{v,a,b\}\}$,
the hypergraph $H_D$ is a point hypergraph.
(Both assertions follow from Theorem~\ref{thm:five} below by checking the tetrahedron condition.)
\end{rem}

\section{Moduli of skew polynomial algebras realizing a point hypergraph}

In this section, we study moduli spaces of skew polynomial algebras realizing
a fixed point hypergraph.  We use the classical language of affine varieties.
An affine variety over $\kk$ means a Zariski closed subset of some affine space over $\kk$; irreducibility is not assumed.
The points of such a variety are understood to be closed points.
Thus no generic points from scheme theory are involved.

Let $H=(V,\mathcal H)$ be a $3$-uniform hypergraph.
We use the notation from Section~\ref{sec:chain}.
After choosing a total order on $V$, we have an isomorphism of affine algebraic groups
\[
C^1(V;\kk^\times)\to (\kk^\times)^{\binom{|V|}{2}};
\quad \varphi \mapsto (\varphi(i,j)\mid i<j).
\]

Recall that $G_H=C_1(V;\mathbb Z)/L_H$, where $L_H=\langle \partial\vv{T}\mid T\in\mathcal H\rangle_{\mathbb Z}$.
The quotient map $C_1(V;\mathbb Z) \to G_H$ induces an embedding
\[
X_H:=\operatorname{Hom}(G_H,\kk^\times)
\hookrightarrow
C^1(V;\kk^\times).
\]
Its image is
$X_H=\{ \varphi\in C^1(V;\kk^\times) \mid \varphi(\partial\vv{T})=1
\text{ for all }T\in\mathcal H\}$.
Since the functions
$C^1(V;\kk^\times) \to \kk^\times; \varphi\mapsto\varphi(\partial\vv{T})$
are regular group homomorphisms, $X_H$ is a closed subgroup of
$C^1(V;\kk^\times)$, and hence an affine algebraic group.

Note that, with $a_{ij}=\varphi(i,j)$ for $i<j$,
the face conditions become $a_{ij}a_{jk}-a_{ik}=0$
for $i<j<k$ with $\{i,j,k\}\in\mathcal H$.
Viewing these conditions as binomial equations,
we can interpret the description of $X_H$ in terms of $G_H$,
including the computation of its dimension below,
as a special case of the theory of Laurent binomial ideals;
see \cite[Section~2]{ES}.

For each non-face $N\in\binom{V}{3}\setminus\mathcal H$,
define the regular group homomorphism
\[
\chi_N:X_H\to\kk^\times;
\quad
\rho \mapsto \rho([\partial\vv{N}]_H).
\]
We define
\[
\mathcal R_H
:=
\left\{
\rho\in X_H
\ \middle|\
\chi_N(\rho)\ne1
\text{ for all }N\in\binom{V}{3}\setminus\mathcal H
\right\}.
\]
Then $\mathcal R_H =
\{\rho\in X_H\mid F_H(\rho)\ne0\}$, where $F_H=
\prod_{N\in\binom{V}{3}\setminus\mathcal H}
(\chi_N-1)$.
Thus $\mathcal R_H$ is the principal open subset of $X_H$ defined by
$F_H \ne 0$.
Since a principal open subset of an affine variety is affine,
$\mathcal R_H$ is an affine variety.

Recall that, for a skew polynomial algebra $S$, we write $H(S)$ for the
$3$-uniform hypergraph associated to the coordinate $2$-planes contained in the point scheme $\operatorname{Pts}(S)$.

\begin{prop}\label{prop:realization}
Let $H=(V,\mathcal H)$ be a $3$-uniform hypergraph.
Then there exists a bijection
\[
\mathcal R_H
\longleftrightarrow
\left\{
\begin{array}{c}
\text{skew polynomial algebras $S$ with variables $x_i$ $(i\in V)$}\\
\text{such that $H(S)=H$}
\end{array}
\right\}.
\]
\end{prop}

\begin{proof}
Let $\rho\in\mathcal R_H$.
Then $\rho\in X_H$, and we regard it as an element of $C^1(V;\kk^\times)$ via the embedding $X_H\hookrightarrow C^1(V;\kk^\times)$.
Since $\rho\in\mathcal R_H$, we have $\rho([\partial\vv{N}]_H)\ne 1$ for every non-face $N\in\binom{V}{3}\setminus\mathcal H$.
We set $r_{ij}= \rho(i,j)$ for distinct $i,j\in V$ and $r_{ii}=1$ for $i \in V$. Define
\[
S_\rho
=\kk\langle x_i\mid i\in V\rangle/(x_ix_j-r_{ij}x_jx_i \mid i,j\in V).
\]
For $T\in\binom{V}{3}$ with $\vv{T}=(i,j,k)$, we have
\[
T\in \mathcal H(S_\rho)
\quad\Longleftrightarrow\quad
r_{ij}r_{jk}r_{ki}=1
\quad\Longleftrightarrow\quad
\rho(\partial\vv{T})=1.
\]
On the other hand, if $T\in\mathcal H$, then $\rho(\partial\vv{T})=1$ because $\rho\in X_H$.  If $T\notin\mathcal H$, then
$\rho(\partial\vv{T})
=\rho([\partial\vv{T}]_H)
=\chi_T(\rho)
\ne 1$
because $\rho\in\mathcal R_H$.
Hence $H(S_\rho)=H$.

Conversely, let
$S_a=\kk\langle x_i\mid i\in V\rangle/
(x_ix_j-a_{ij}x_jx_i \mid i,j\in V,\ i\ne j)$
be a skew polynomial algebra with $H(S_a)=H$.
Define $\varphi_a\in C^1(V;\kk^\times)$ by $\varphi_a(i,j)=a_{ij}$.
For $T\in\mathcal H =\mathcal H(S_a)$, we have $\varphi_a(\partial\vv{T})=1$. 
This says that $\varphi_a$ is trivial on $L_H$, so it factors
uniquely through $G_H$.
Thus there exists a unique homomorphism
$\rho_a:G_H\to\kk^\times$
such that
\[
\rho_a([c]_H)=\varphi_a(c)
\]
for $c\in C_1(V;\mathbb Z)$.
Furthermore for $N\in\binom{V}{3}\setminus\mathcal H=\binom{V}{3}\setminus\mathcal H(S_a)$, we have
$\varphi_a(\partial\vv{N})\ne1$.
This says that
$\chi_N(\rho_a)\ne1$ for all $N\in\binom{V}{3}\setminus\mathcal H$.
Hence $\rho_a\in\mathcal R_H$.

The assignments
$\rho\mapsto S_\rho$ and $\rho_a \mapsfrom S_a$
are inverse to each other.
Hence the desired bijection follows.
\end{proof}

We next construct a moduli space whose points correspond to graded
isomorphism classes of skew polynomial algebras realizing $H$.
We use the following known criterion.

\begin{prop}[{\cite[Lemma 2.3]{Vi}, \cite[Theorem 2.4]{Ga}}]\label{prop.skewiso}
Let $S_{a}$ and $S_{a'}$ be skew polynomial algebras in $n$ variables generated in degree $1$.
Then $S_{a} \cong S_{a'}$ as graded algebras if and only if  there exists a permutation $\sigma \in {\mathfrak S}_n$ such that $a'_{ij}=a_{\sigma(i)\sigma(j)}$ for $1 \leq i,j \leq n$ with $i\neq j$. 
\end{prop}

Let $H=(V,\mathcal H)$ be a $3$-uniform hypergraph, and set
\[
\Gamma:=\operatorname{Aut}(H)
=\{\sigma\in{\mathfrak S}_V\mid \sigma(\mathcal H)=\mathcal H\}.
\]
For $\sigma\in\Gamma$, define an automorphism
$\sigma_*:C_1(V;\mathbb Z)\to C_1(V;\mathbb Z)$
by
$\sigma_*((i,j))=(\sigma(i),\sigma(j))$.
Since $\sigma(\mathcal H)=\mathcal H$, we have
$\sigma_*(L_H)=L_H$.
Thus $\sigma_*$ induces an automorphism
$\bar\sigma:G_H\to G_H$.
We define an action of $\Gamma$ on
$X_H=\operatorname{Hom}(G_H,\kk^\times)$
by
\[
(\sigma\cdot\rho)(g)=\rho(\bar\sigma^{-1}g)
\]
for $g\in G_H$.
This action preserves $\mathcal R_H$.  Indeed, let
$\rho\in\mathcal R_H$ and let $N\in\binom{V}{3}\setminus\mathcal H$.
Choose an orientation $\vv{N}$ of $N$.  Since $\sigma^{-1}(N)$ is again a
non-face of $H$, we have
$\rho\bigl([\partial(\sigma^{-1}\vv{N})]_H\bigr)\ne1$.
Thus it follows that
\[
(\sigma\cdot\rho)([\partial\vv{N}]_H)
\,=\,\rho\bigl(\bar\sigma^{-1}[\partial\vv{N}]_H\bigr)
\,=\,\rho\bigl([\partial(\sigma^{-1}\vv{N})]_H\bigr)
\,\ne\,1.
\]
Therefore $\sigma\cdot\rho\in\mathcal R_H$.

Let $A_H:=\kk[\mathcal R_H]$ be the coordinate ring of $\mathcal R_H$.
The action of $\Gamma$ on $\mathcal R_H$ induces an action on $A_H$ by
$(\sigma\cdot f)(\rho)=f(\sigma^{-1}\cdot\rho)$,
where $\sigma\in\Gamma,  f\in A_H,  \rho\in\mathcal R_H$.
Let
\[
A_H^\Gamma :=\{f\in A_H \mid \sigma\cdot f=f\text{ for all }\sigma\in\Gamma\}
\]
be the invariant subring.
Since $\Gamma$ is finite, $A_H^\Gamma$ is a finitely generated reduced $\kk$-algebra. We define
\[
\mathcal M_H
:= \operatorname{MaxSpec}(A_H^\Gamma).
\]
Thus $\mathcal M_H$ is an affine variety.
The inclusion $A_H^\Gamma\hookrightarrow A_H$
induces a morphism
\[
q:\mathcal R_H \to \mathcal M_H.
\]
Since $\Gamma$ is finite and $\mathcal R_H$ is affine, this morphism is a geometric quotient by $\Gamma$; see \cite[Example 6.1]{Do}.
Hence, for $\rho,\rho'\in\mathcal R_H$, it follows that
$q(\rho)=q(\rho')$ if and only if $\rho'=\sigma\cdot\rho$ for some $\sigma\in\Gamma$.
Thus the closed points of $\mathcal M_H$ are naturally identified with the $\Gamma$-orbits in $\mathcal R_H$.

\begin{prop}\label{prop:moduli-iso}
Let $H=(V,\mathcal H)$ be a $3$-uniform hypergraph.
Then there is a natural bijection
\[
\mathcal M_H
\longleftrightarrow
\left\{
\begin{array}{c}
\text{graded isomorphism classes of skew polynomial algebras $S$}\\
\text{such that $H(S)=H$}
\end{array}
\right\}.
\]
\end{prop}

\begin{proof}
Since $\Gamma$ is finite, $A_H$ is integral over $A_H^\Gamma$.
Hence the morphism $q:\mathcal R_H\to\mathcal M_H$
is surjective.  Therefore every point of $\mathcal M_H$ is of the form
$q(\rho)$ for some $\rho\in\mathcal R_H$.

We now check that $q(\rho) \mapsto [S_\rho]$ gives the desired bijection, 
where $[S_\rho]$ denotes the graded isomorphism class of $S_\rho$.
This is well-defined.  Indeed, suppose that $q(\rho)=q(\rho')$.
Then $\rho'=\sigma\cdot\rho$ for some $\sigma\in\Gamma$.
Let $r_{ij}=\rho(i,j), r'_{ij}=\rho'(i,j)$.
By the definition of the action, we have $r'_{ij}=r_{\sigma^{-1}(i)\sigma^{-1}(j)}$
for distinct $i,j\in V$.
Hence Proposition~\ref{prop.skewiso} gives
$S_\rho\cong S_{\rho'}$
as graded algebras.

The map is surjective by Proposition~\ref{prop:realization}.
It remains to prove injectivity.
Suppose $S_\rho\cong S_{\rho'}$ as graded algebras.
By Proposition~\ref{prop.skewiso}, there exists $\sigma\in{\mathfrak S}_V$
such that
\begin{equation}\label{eq:r'}
r'_{ij}=r_{\sigma(i)\sigma(j)}
\end{equation}
for distinct  $i,j \in V$.
We show that $\sigma\in\Gamma$.
Let $T=\{i,j,k\}\in\binom{V}{3}$ and choose the orientation
$\vv{T}=(i,j,k)$.
Since $H=H(S_\rho)=H(S_{\rho'})$, we have
\begin{align*}
T \in \mathcal H
&\;\;  \Longleftrightarrow \;\;
T \in \mathcal H(S_{\rho'})
\;\;  \Longleftrightarrow \;\;
r'_{ij}r'_{jk}r'_{ki}=1 \\
&\;\;  \overset{\text{\eqref{eq:r'}}}{\Longleftrightarrow} \;\;
r_{\sigma(i)\sigma(j)}r_{\sigma(j)\sigma(k)}r_{\sigma(k)\sigma(i)}=1
\;\;  \Longleftrightarrow \;\;
\sigma(T) \in \mathcal H(S_\rho)
\;\; \Longleftrightarrow \;\;
\sigma(T) \in \mathcal H.
\end{align*}
Thus $\sigma(\mathcal H)=\mathcal H$, so $\sigma\in\Gamma$.
By \eqref{eq:r'}, we have $\rho'(i,j)=\rho(\sigma(i),\sigma(j))$ for distinct $i,j \in V$.
This says that $\rho'=\sigma^{-1}\cdot\rho$.
Since $\sigma^{-1}\in \Gamma$, we obtain $q(\rho)=q(\rho')$.
\end{proof}

\begin{thm}\label{thm:moduli-dim}
Let $H=(V,\mathcal H)$ be a point hypergraph.
Then
\[
\dim \mathcal M_H = \dim \mathcal R_H =\operatorname{rank}_{\mathbb Z}G_H.
\]
\end{thm}

\begin{proof}
Since $G_H$ is a finitely generated abelian group, we have an
isomorphism
$G_H \cong \mathbb Z^r\oplus T$,
where $r=\operatorname{rank}_{\mathbb Z}G_H$ and $T$ is a finite abelian group.
This gives an isomorphism of affine varieties
\[
X_H=\operatorname{Hom}(G_H,\kk^\times)
\cong
\operatorname{Hom}(\mathbb Z^r,\kk^\times)
\times \operatorname{Hom}(T,\kk^\times).
\]
Since $\operatorname{Hom}(\mathbb Z^r,\kk^\times)\cong(\kk^\times)^r$
and $\operatorname{Hom}(T,\kk^\times)$ is finite,
$X_H$ is a finite union of $r$-dimensional algebraic tori.
Hence $\dim X_H=r$.

Since $H$ is a point hypergraph, $\mathcal R_H$ is a nonempty principal open subset of $X_H$, so $\dim \mathcal R_H=\dim X_H=r$.

Finally, recall that $\mathcal M_H=\operatorname{MaxSpec}(A_H^\Gamma)$,
where $A_H=\kk[\mathcal R_H]$ and $\Gamma=\operatorname{Aut}(H)$.
Since $\Gamma$ is finite, $A_H$ is integral over $A_H^\Gamma$.
Thus $A_H$ and $A_H^\Gamma$ have the same Krull dimension. Hence $\dim \mathcal M_H=\dim \mathcal R_H=r$.
This proves the claim.
\end{proof}

We now apply Theorem~\ref{thm:moduli-dim} to point hypergraphs arising from
triangulations of orientable closed surfaces and determine the dimensions of
their moduli spaces explicitly.

\begin{thm}\label{thm:surface-moduli}
Let $K$ be a finite triangulation of a connected orientable
closed surface of Euler characteristic $\chi$.
Let $D\subset K_2$, set $d=|D|$, and suppose that
$H_D=(V(K),K_2\setminus D)$ is a point hypergraph.
Then $d\ne1$ and
\[
\dim \mathcal M_{H_D}
=
\begin{cases}
\displaystyle \binom{|V(K)|}{2}-2|V(K)|+2\chi+1, & d=0,\\[2mm]
\displaystyle \binom{|V(K)|}{2}-2|V(K)|+2\chi+d, & d\ge2.
\end{cases}
\]
\end{thm}

\begin{proof}
Put $f_i=|K_i|$.
Since $H_D$ is a point hypergraph, Theorem~\ref{thm:moduli-dim} yields
$\dim \mathcal M_{H_D}
=\operatorname{rank}_{\mathbb Z}G_{H_D}$.
Moreover, $d\ne1$ by Theorem~\ref{thm:surface-obst}.

Since $G_{H_D}=C_1(V(K);\mathbb Z)/L_{H_D}$ and 
$\operatorname{rank}_{\mathbb Z} C_1(V(K);\mathbb Z) = \binom{f_0}{2}$, we have
\[
\operatorname{rank}_{\mathbb Z}G_{H_D}
=\binom{f_0}{2}-\operatorname{rank}_{\mathbb Z}L_{H_D}.
\]
Let $\partial_K:C_2(K;\mathbb Q)\to C_1(K;\mathbb Q)$ be the simplicial boundary map of $K$.

First, suppose that $d=0$.  Then $H_D=(V(K),K_2)$ and
$L_{H_D}=\langle \partial \vv{T} \mid T\in K_2 \rangle_{\mathbb Z}$, so we have
\[
\operatorname{rank}_{\mathbb Z}L_{H_D}
\;=\; \dim_{\mathbb Q} \langle \partial \vv{T} \mid T\in K_2 \rangle_{\mathbb Q}
\;=\; \dim_{\mathbb Q} \operatorname{Im} \partial_K
\;=\; \dim_{\mathbb Q}C_2(K;\mathbb Q)-\dim_{\mathbb Q} \operatorname{Ker} \partial_K.
\]
Since $K$ is a connected orientable closed surface, we have
$\operatorname{Ker} \partial_K=H_2(K;\mathbb Q)\cong\mathbb Q$.
Moreover $\dim_{\mathbb Q}C_2(K;\mathbb Q)=f_2$, so it follows that
\[
\operatorname{rank}_{\mathbb Z}L_{H_D}=f_2-1.
\]

Next, suppose that $d\ge2$.  Since
$L_{H_D}=\langle \partial\vv{T}\mid T\in K_2\setminus D\rangle_{\mathbb Z}$, we have
\[
\operatorname{rank}_{\mathbb Z}L_{H_D}
=
\dim_{\mathbb Q}
\langle \partial\vv{T}\mid T\in K_2\setminus D\rangle_{\mathbb Q}.
\]
We claim that the vectors
$\{\partial\vv{T}\mid T\in K_2\setminus D\}$
are linearly independent over $\mathbb Q$.
Indeed, suppose that
$\sum_{T\in K_2\setminus D}\lambda_T\partial\vv{T}=0$
for some $\lambda_T\in\mathbb Q$, and set
$z=\sum_{T\in K_2\setminus D}\lambda_T\vv{T}$.
Then $z\in\operatorname{Ker}\partial_K$.
As above, $\operatorname{Ker}\partial_K=H_2(K;\mathbb Q)\cong\mathbb Q$,
and this kernel is generated by a fundamental $2$-cycle with support $K_2$.
Since $z$ is supported on $K_2\setminus D\subsetneq K_2$, it follows that
$z=0$, and hence $\lambda_T=0$ for every $T\in K_2\setminus D$.
Thus the vectors are linearly independent, and it follows that
\[
\operatorname{rank}_{\mathbb Z}L_{H_D}=f_2-d.
\]

Finally, each edge of $K$ is contained in exactly two triangles, so
$3f_2=2f_1$.
Combining this with the Euler formula $f_0-f_1+f_2=\chi$, we get $f_2=2f_0-2\chi$.
Therefore
\[
\dim\mathcal M_{H_D}
=\operatorname{rank}_{\mathbb Z}G_{H_D}
=
\begin{cases}
\displaystyle
\binom{f_0}{2}-f_2+1
=\binom{f_0}{2}-2f_0+2\chi+1,
& d=0,\\[2mm]
\displaystyle
\binom{f_0}{2}-f_2+d
=\binom{f_0}{2}-2f_0+2\chi+d,
& d\ge2.
\end{cases}
\]
Hence the assertion follows.
\end{proof}

Note that the Euler characteristic of a connected orientable closed surface of genus $g$ is given by $\chi=2-2g$.

\begin{ex}\label{ex:octahedral-moduli} 
Let $K$ be the boundary complex of the octahedron.
We label the vertices so that the pairs of opposite vertices are
$\{1,6\}$, $\{2,5\}$, and $\{3,4\}$.
Thus
\[
K_2=\{123,124,135,145,236,246,356,456\}.
\]
Set $H^{+}=(V(K),K_2)$.
By Proposition~\ref{prop:realization}, a point of $\mathcal R_{H^{+}}$
corresponds to a skew polynomial algebra
\[
S_a=\kk\langle x_1,\dots,x_6\rangle/ (x_ix_j-a_{ij}x_jx_i \mid 1\le i, j\le 6)
\]
with $H(S_a)=H^{+}$.  Note that $a_{ji}=a_{ij}^{-1}$ and $a_{ii}=1$.
Take
\[
a=a_{12},\quad b=a_{13},\quad c=a_{14},\quad d=a_{15},\quad e=a_{26}.
\]
Then the face conditions $a_{ij}a_{jk}a_{ki}=1$ determine the following parameters:
\[
a_{23}=\frac ba,\quad
a_{24}=\frac ca,\quad
a_{35}=\frac db,\quad
a_{45}=\frac dc,\quad
a_{36}=\frac{ae}{b},\quad
a_{46}=\frac{ae}{c},\quad
a_{56}=\frac{ae}{d}.
\]
Since $16, 25, 34$ are non-edges of $K$, the three parameters
$a_{16}, a_{25}, a_{34}$ are free before imposing the non-face conditions.
We write
\[
a_{16}=\alpha ae,\qquad
a_{25}=\frac{\beta d}{a},\qquad
a_{34}=\frac{\gamma c}{b}.
\]
Then the twelve non-face conditions reduce to
\[
\alpha\ne1,\qquad \beta\ne1,\;\;\text{ and }\;\; \gamma\ne1.
\] 
Thus a point of $\mathcal R_{H^{+}}$ is uniquely determined by a tuple
$(a,b,c,d,e,\alpha,\beta,\gamma)\in(\kk^\times)^8$ satisfying
$\alpha,\beta,\gamma\ne1$.  Conversely, any such tuple gives parameters
$a_{ij}$ satisfying exactly the required face and non-face conditions.
Therefore
\[
\mathcal R_{H^{+}}
\cong
(\kk^\times)^5\times(\kk^\times\setminus\{1\})^3,
\]
and hence $\dim \mathcal R_{H^{+}}=8$.

Now we look at the corresponding moduli space.
The automorphism group of the octahedral boundary is the group of
permutations preserving the set of the three pairs of opposite vertices
$\{1,6\}$, $\{2,5\}$, and $\{3,4\}$.
Hence
$\operatorname{Aut}(H^{+})
\cong
(\mathbb Z/2\mathbb Z)^3\rtimes{\mathfrak S}_3$
and
$|\operatorname{Aut}(H^{+})|=48$.
The quotient map
$q:\mathcal R_{H^{+}}\to\mathcal M_{H^{+}}$
sends two points $\rho,\rho'\in\mathcal R_{H^{+}}$ to the same point of
$\mathcal M_{H^{+}}$ if and only if
$\rho'=\sigma\cdot\rho$ for some $\sigma\in\operatorname{Aut}(H^{+})$.
Thus, for a general point $\rho\in\mathcal R_{H^{+}}$, the fiber
$q^{-1}(q(\rho))$ has $48$ points.
If $\rho$ has nontrivial stabilizer in $\operatorname{Aut}(H^{+})$, then this
fiber has fewer points.
Since $\operatorname{Aut}(H^{+})$ is finite, the quotient does not change the
dimension.  Therefore
\[
\dim \mathcal M_{H^{+}}
=\dim \mathcal R_{H^{+}}
=8.
\]
This agrees with Theorem~\ref{thm:surface-moduli}, since $\chi=2$, $d=0$, and
$\binom{6}{2}-2\cdot6+4+1=8$.
\end{ex}

\section{The six-vertex case}

In this section, we characterize which $3$-uniform hypergraphs on six vertices are point hypergraphs.

\begin{dfn}\label{def:tetra}
Let $H$ be a $3$-uniform hypergraph.
We say that $H$ satisfies the \textit{tetrahedron condition} if no induced
sub-hypergraph of $H$ is isomorphic to the boundary of a tetrahedron with one
triangle deleted.
\end{dfn}

The boundary of a tetrahedron with one triangle deleted
is illustrated below.
\begin{center}
\begin{tikzpicture}[
  scale=0.9,
  edge/.style={draw=black,line width=0.65pt,line join=round},
  vertex/.style={circle,fill=black,inner sep=1.2pt},
  every node/.style={font=\small}
]
  \coordinate (I) at (0,1.5);
  \coordinate (J) at (-1.1,-0.6);
  \coordinate (K) at (0.15,-0.1);
  \coordinate (L) at (1.2,-0.45);

  \path[edge,fill=gray!20] (I)--(K)--(L)--cycle;
  \path[edge,fill=gray!38] (J)--(K)--(L)--cycle;
  \draw[edge] (I)--(J)--(K)--cycle;

  \foreach \p in {I,J,K,L}
    \node[vertex] at (\p) {};
  \node[above right] at (I) {$i$};
  \node[below left] at (J) {$j$};
  \node[above left,inner sep=1.5pt] at (K) {$k$};
  \node[right] at (L) {$\ell$};

  \node (D) at (-1.05,1.85) {deleted $ijk$};
  \draw[->,line width=0.45pt] (D.south) -- (-0.32,0.40);
\end{tikzpicture}
\end{center}
Thus, the tetrahedron condition means that $H$ has no
induced sub-hypergraph $(W,\mathcal F)$ with
$W=\{i,j,k,\ell\}$ and $\mathcal F=\{ij\ell,ik\ell,jk\ell \}$
(see also Example~\ref{ex:tetrahedron}).

\begin{dfn}\label{def:octa}
A $3$-uniform hypergraph is called an \textit{octahedral obstruction} if it is
isomorphic to the boundary of an octahedron with one triangle deleted.
\end{dfn}

The boundary of an octahedron with one triangle deleted
is illustrated below.
\begin{center}
\begin{tikzpicture}[
  scale=0.9,
  edge/.style={draw=black,line width=0.65pt,line join=round},
  hidden/.style={draw=black!45,densely dashed,line width=0.45pt},
  vertex/.style={circle,fill=black,inner sep=1.2pt},
  every node/.style={font=\small}
]
  \coordinate (A) at (0,1.5);
  \coordinate (B) at (-1.1,-0.15);
  \coordinate (C) at (0.35,-0.45);
  \coordinate (D) at (-0.35,0.45);
  \coordinate (E) at (1.1,0.15);
  \coordinate (F) at (0,-1.5);

  \draw[hidden] (A)--(D)--(B);
  \draw[hidden] (E)--(D)--(F);
  \path[edge,fill=gray!20] (A)--(C)--(E)--cycle;
  \path[edge,fill=gray!38] (B)--(C)--(F)--cycle;
  \path[edge,fill=gray!27] (C)--(E)--(F)--cycle;
  \draw[edge] (A)--(B)--(C)--cycle;

  \foreach \p in {A,B,C,E,F}
    \node[vertex] at (\p) {};
  \node[circle,fill=black!45,inner sep=1.1pt] at (D) {};
  \node[above right] at (A) {$1$};
  \node[left] at (B) {$2$};
  \node[right] at (C) {$3$};
  \node[below left,text=black!60] at (D) {$4$};
  \node[right] at (E) {$5$};
  \node[below] at (F) {$6$};

  \node (del) at (-1.05,1.85) {deleted $123$};
  \draw[->,line width=0.45pt] (del.south) -- (-0.02,0.38);
\end{tikzpicture}
\end{center}
In other words, an octahedral obstruction is a hypergraph
isomorphic to $(V,\mathcal H)$, where
$V=\{1,2,3,4,5,6\}$ and $\mathcal H=\{124,135,145,236,246,356,456\}$.

\begin{prop}\label{prop:tetraocta}
Let $H=(V,\mathcal H)$ be a $3$-uniform hypergraph with $|V|=6$.
If $H$ is a point hypergraph, then it satisfies the tetrahedron condition and is not an octahedral obstruction.
\end{prop}

\begin{proof}
The boundaries of a tetrahedron and of an octahedron are triangulations of $S^2$.
Thus the assertion follows from Corollary~\ref{cor:induced-obst}.
\end{proof}

The rest of the section is devoted to the converse; see Theorem~\ref{thm:tetraocta}.
For comparison, we first recall the five-vertex case.

\begin{thm}\label{thm:five}
Let $H=(V,\mathcal H)$ be a $3$-uniform hypergraph with $|V|=5$.
Then $H$ is a point hypergraph if and only if it satisfies the tetrahedron condition.
\end{thm}

\begin{proof}
The only-if direction follows from Corollary~\ref{cor:induced-obst}, since the
boundary of a tetrahedron is a triangulation of $S^2$.
We prove the converse. The five-variable classification in
\cite[Fig.~2 in Section~4.3]{BDL} shows that all sixteen
isomorphism classes of adequate collections are realized as collections of coordinate planes not contained in the point variety of a quantum polynomial algebra.
Since adequacy is precisely the tetrahedron condition
in the language of $3$-uniform hypergraphs,
the assertion follows.
\end{proof}

The following closure operation keeps track of the faces forced by the tetrahedron condition.

\begin{dfn}\label{def:tetraclosure}
Let $V$ be a finite set.
A subset $A\subset \binom{V}{3}$ is called
\textit{tetrahedrally closed} if, for every four-element subset $Q\subset V$,
whenever three triangles of $\binom{Q}{3}$ belong to $A$, the remaining one also belongs to $A$.

For $S\subset \binom{V}{3}$, the \textit{tetrahedral closure} of $S$, denoted by $\operatorname{cl}(S)$, is the smallest tetrahedrally closed subset of $\binom{V}{3}$ containing $S$.
\end{dfn}

The tetrahedron condition ensures that the tetrahedral closure of any set of faces remains inside the face set.

\begin{lem}\label{lem:tetraclosure}
Suppose that $H=(V,\mathcal H)$ satisfies the tetrahedron condition.
Let $S\subset\mathcal H$.
Then $\operatorname{cl}(S)\subset\mathcal H$.
In particular, if $N\notin\mathcal H$, then $N\notin\operatorname{cl}(S)$.
\end{lem}

\begin{proof}
Since $H$ satisfies the tetrahedron condition, the face set $\mathcal H$ is
tetrahedrally closed.  Since $S\subset\mathcal H$ and $\operatorname{cl}(S)$ is
the smallest tetrahedrally closed subset of $\binom{V}{3}$ containing $S$, we
have $\operatorname{cl}(S)\subset\mathcal H$.
The final assertion follows immediately.
\end{proof}

We next introduce the combinatorial objects which remain to be analyzed.

\begin{dfn}\label{def:pointed-circuit}
Let $V$ be a finite set. A subset $C\subset \binom{V}{3}$ is called a \textit{circuit} if the vectors
\[ \{\partial\vv{T}\mid T\in C\} \]
are linearly dependent over $\mathbb Q$, and every proper subcollection is linearly independent.

A \textit{pointed circuit} is a pair $(C,N)$, where $C$ is a circuit and $N\in C$.
\end{dfn}

The circuits defined above are the circuits of the full simplicial
matroid on $\binom{V}{3}$ over $\mathbb Q$;
see \cite[Definition~6.1.1]{CL}.
For example, the triangles of any finite triangulation of
a connected orientable closed surface form a circuit;
see \cite[Example~6.1.4]{CL}.

We now show that a rationally trivial non-face produces a pointed circuit $(C,N)$ such that $N\notin\operatorname{cl}(C\setminus\{N\})$.

\begin{lem}\label{lem:bad-nonface}
Let $H=(V,\mathcal H)$ be a $3$-uniform hypergraph  satisfying the tetrahedron condition.
Let $N\notin\mathcal H$ be a non-face such that
\[
[\partial\vv{N}]_H\otimes 1=0
\quad\text{in }G_H\otimes_{\mathbb Z}\mathbb Q.
\]
Then there exists a pointed circuit $(C,N)$ such that
$C\setminus\{N\}\subset\mathcal H$ and $N\notin\operatorname{cl}(C\setminus\{N\})$.
\end{lem}

\begin{proof}
Since
$[\partial\vv{N}]_H\otimes 1=0$ in $G_H\otimes_{\mathbb Z}\mathbb Q$,
there exists a finite subset $S\subset\mathcal H$ such that
\[
\partial\vv{N}\in
\left\langle \partial\vv{T}\mid T\in S\right\rangle_{\mathbb Q}.
\]
Choose such an $S$ minimal with respect to inclusion, and set $C:=S\cup\{N\}$.
The minimality of $S$ implies that $C$ is a circuit.
Moreover, $C\setminus\{N\}\subset\mathcal H$.
By Lemma~\ref{lem:tetraclosure}, $\operatorname{cl}(S)\subset\mathcal H$.
Since $N\notin\mathcal H$, we have $N\notin\operatorname{cl}(S)$.
\end{proof}

Therefore, to understand when a rationally trivial non-face can occur, we analyze pointed circuits $(C,N)$ satisfying
$N\notin\operatorname{cl}(C\setminus\{N\})$.
Before the classification lemma, we recall the following elementary facts. 
In what follows, let $K$ be a finite two-dimensional simplicial complex and fix the following notation.  
\begin{itemize}
\item Let $v,e$, and $f$ be the numbers of vertices, edges, and triangles of $K$, respectively. 
\item For each edge $E$ of $K$, let $m(E)$ denote the number of triangles of $K$ containing $E$, and let us call it the \emph{degree} of $E$.  
\end{itemize}
\begin{lem}\label{lem:2e=3f}
Suppose that every edge of $K$ is contained in at least two triangles and that $2e=3f$. 
Then every edge of $K$ is contained in exactly two triangles, and, for every vertex $x$ of $K$, the link $\operatorname{lk}_K(x)$ is a disjoint union of cycles. 
\end{lem}
\begin{proof}
We have
\[\sum_{E : \text{ edges of }K} m(E)=3f.\] 
On the other hand, by the assumption, we also have $m(E)\geq 2$ for every edge $E$. 
Since $3f=2e$, it follows that
\[\sum_{E : \text{ edges of }K}\bigl(m(E)-2\bigr)=3f-2e=0.\]
Each summand is a nonnegative integer, and hence $m(E)=2$ holds for every edge $E$ of $K$. 

Now, fix a vertex $x$ of $K$. Then the link $\operatorname{lk}_K(x)$ is a finite graph. 
Let $y$ be one of its vertices. Then $\{x,y\}$ is an edge of $K$, 
and the vertices adjacent to $y$ in $\operatorname{lk}_K(x)$ are precisely the vertices $z$ such that $\{x,y,z\}$ is a triangle of $K$. 
Therefore, $\deg_{\operatorname{lk}_K(x)}(y)=m(\{x,y\})=2$. This implies that $\operatorname{lk}_K(x)$ is a disjoint union of cycles. 
\end{proof}

\begin{lem}\label{lem:shakehand} 
For every vertex $x$ of $K$, the number of edges of odd degree incident with $x$ is even. 
In particular, if $m(E) \geq 2$ for every edge $E$ in $K$, then $3f - 2e=1$ never happens. 
\end{lem}

\begin{proof}
Fix a vertex $x$ of $K$ and consider the graph $L=\operatorname{lk}_K(x)$. 
The vertices of $L$ are in bijection with the edges of $K$ incident with $x$.
Moreover,
\[\deg_L(y)=m(\{x,y\})\]
for every vertex $y \in L$.
Thus, the assertion directly follows from the handshaking lemma for $L$. 

For the last assertion, suppose that $m(E)\ge 2$ for every edge $E$ and that $3f-2e=1$.  Since
\[
\sum_{E : \text{ edges of }K}(m(E)-2)=3f-2e=1,
\]
there is exactly one edge with degree $3$ and every other edge has degree $2$. 
This contradicts the first assertion.
Hence $3f-2e=1$ cannot occur.
\end{proof}

\begin{lem}\label{lem:link-surface}
Let $K$ be a finite connected $2$-dimensional simplicial complex.
Suppose that, for every vertex $x$, the link $\operatorname{lk}_K(x)$ is a cycle. Then $K$ is a triangulation of a closed surface. 
\end{lem}

\begin{proof}
By the standard local criterion for PL manifolds
\cite[Corollary~1.16 and the following remark]{Hu}, the assumption on vertex links implies that $|K|$ is a PL $2$-manifold without boundary.
Since $K$ is finite and connected, $|K|$ is a closed surface.
\end{proof}

\begin{lem}\label{lem:circuit}
Let $V=\{1,2,3,4,5,6\}$ and let $C\subset\binom{V}{3}$ be a circuit.
Then, up to relabeling of vertices, $C$ is one of the following six types.
In the last column, we write $ijk$ for $\partial\vv{ijk}$, where each triangle
is oriented in the increasing order of its vertices.
\[
\resizebox{\textwidth}{!}{$
\begin{array}{c@{\qquad}l@{\qquad}l}
\toprule
\text{type} & C & \text{primitive relation} \\
\midrule
4 &
\{123,124,134,234\} &
123-124+134-234=0 \\
\addlinespace[1mm]

6 &
\{123,124,134,235,245,345\} &
123-124+134-235+245-345=0 \\
\addlinespace[1mm]

8\mathrm a &
\{123,124,134,235,245,346,356,456\} &
123-124+134-235+245-346+356-456=0 \\
\addlinespace[1mm]

8\mathrm b &
\{123,124,135,145,236,246,356,456\} &
123-124+135-145-236+246-356+456=0 \\
\addlinespace[1mm]

10 &
\{123,124,125,126,134,156,235,246,345,456\} &
123-124-125+126+134-156-235+246-345+456=0 \\
\addlinespace[1mm]

11 &
\{123,124,125,134,136,156,235,236,246,345,456\} &
2\cdot123-124-125+134+136-156-235-236+246-345+456=0
\\
\bottomrule
\end{array}
$}
\]
\end{lem}
\begin{proof}
Let $K_C=\{\sigma:\sigma\subseteq T\text{ for some }T\in C\}$ be the simplicial complex generated by the triangles in $C$. 
Since $C$ is a circuit, there is a relation
\begin{equation}\label{e:cirrel}
\sum_{T\in C}\lambda_T\partial\vv{T}=0,\qquad \lambda_T\in\mathbb Q^\times.
\end{equation}
The minimality of $C$ implies that the $1$-skeleton of $K_C$, denoted by $K_C^{(1)}$, is connected.
Moreover, for an edge $E$ of $K_C$, we have $m(E)\geq 2$. 
Indeed, if $E$ were contained in a unique triangle $T\in C$, then the coefficient of $E$ in a circuit relation would be $\pm\lambda_T$, contradicting $\lambda_T\neq 0$. 

Let $v,e,f$ denote the numbers of vertices, edges, and triangles of $K_C$, respectively; thus $f=|C|$. 
It is clear that $4\leq v\leq 6$ and $2e\leq 3f$. 
On the other hand, the vectors $\partial\vv{T}$ are $1$-cycles in $K_C^{(1)}$, whose cycle space has dimension $e-v+1$. 
Hence, $f-1\leq e-v+1$. Therefore, we obtain 
\begin{equation}\label{eq:ef}
f+v-2\leq e\leq \frac{3f}{2},
\qquad\text{and hence}\qquad
f\geq 2v-4.
\end{equation}

\underline{The case $v=4$.}\ \ 
In this case, \eqref{eq:ef} gives $f\geq 4$. Since there are only four
triangles on four vertices, we have $f=4$. Thus $K_C$ is the boundary
of a tetrahedron, and $C$ is of type $4$.

\underline{The case $v=5$.}\ \ 
In this case, \eqref{eq:ef} gives $f\geq 6$. Since the cycle space of the complete graph on five vertices has dimension $\binom{5}{2}-5+1=6$, 
we also have $f-1\leq 6$. Thus, $f=6$ or $7$. 

First, suppose that $f=6$. Then \eqref{eq:ef} implies that $e=9$, and hence $2e=3f$. 
By Lemma~\ref{lem:2e=3f}, every edge of $K_C$ is contained in exactly two triangles and every vertex link is a disjoint union of cycles. 
Since a vertex link has at most four vertices by $v=5$, it must be a single cycle. Thus, by Lemma~\ref{lem:link-surface}, $K_C$ is a triangulation of a closed surface. 
Moreover, \[\chi(K_C)=v-e+f=5-9+6=2,\] 
where $\chi$ denotes the Euler characteristic, so $K_C$ is a triangulation of the $2$-sphere. 
Up to relabeling, there is a unique triangulation of the $2$-sphere with five vertices, and the corresponding circuit $C$ is of type $6$.

Next, suppose that $f=7$. Then \eqref{eq:ef} gives $e=10$, and hence $3f-2e=21-20=1$. 
This contradicts Lemma~\ref{lem:shakehand}. Therefore, $f=7$ cannot occur. 

\underline{The case $v=6$.}\ \ 
In this case, \eqref{eq:ef} gives $f\geq 8$. Since the cycle space of the complete graph on six vertices has dimension $\binom{6}{2}-6+1=10$, 
we also have $f-1\leq 10$. Thus, $f\in\{8,9,10,11\}$.

First, suppose that $f=8$. Then \eqref{eq:ef} implies that $e=12$, and hence $2e=3f$. As in the case $(v,f)=(5,6)$, every vertex link is a single cycle. 
Therefore, $K_C$ is a triangulation of a closed surface, and $\chi(K_C)=6-12+8=2$. 
Thus, $K_C$ is a triangulation of the $2$-sphere. Up to relabeling, there are exactly two triangulations of the $2$-sphere with six vertices, 
which correspond to $C$ of types $8\mathrm{a}$ and $8\mathrm{b}$. 

Next, suppose that $f=9$. Then \eqref{eq:ef} implies that $e=13$, and hence $3f-2e=27-26=1$. 
Again, this contradicts Lemma~\ref{lem:shakehand}. Therefore $f=9$ cannot occur. 

Next, consider the case $f=10$. By \eqref{eq:ef}, we have $e=14$ or $e=15$. 
\begin{itemize}
\item Let us consider the case $e=14$. Then we have
\[\sum_E\bigl(m(E)-2\bigr)=3f-2e=2.\]
Thus either two edges have degree $3$ and all other edges have degree $2$, 
or one edge has degree $4$ and all other edges have degree $2$. 
The former case is impossible by Lemma~\ref{lem:shakehand}, since the graph formed by the odd-degree edges would have exactly two edges 
but every vertex of this graph must have even degree, which is impossible.  
Hence, there is a unique edge of degree $4$. Up to relabeling, we may assume that this edge is $12$. 
Then \[123,124,125,126\in C.\]
Since all edges incident with $1$ or $2$ already occur, the unique missing edge lies in $\{3,4,5,6\}$. 
Up to relabeling, we may assume that the missing edge is $36$. 

Consider the link of the vertex $1$. It already contains the four edges $23,24,25,26$. 
The vertex $2$ has degree $4$ in this link, whereas each of the vertices $3,4,5,6$ must have degree $2$. 
Hence, the remaining edges of $\operatorname{lk}_{K_C}(1)$ form a perfect matching on $\{3,4,5,6\}$. The possible matchings are $\{34,56\}$ and $\{35,46\}$. 
These two possibilities are equivalent under interchanging $4$ and $5$, so we may assume that the matching is $\{34,56\}$. Thus, \[134,156\in C.\]

The link of the vertex $2$ must similarly be completed by one of the same two matchings. 
If it were completed by $\{34,56\}$, then $234,256\in C$. 
At this point, every edge incident with $1$ or $2$ would already have its prescribed degree, so the remaining two triangles would have to lie on $\{3,4,5,6\}$. 
Since $36$ is missing, the only possibilities would be $345$ and $456$. However, these contain the edges $34$ and $56$, respectively, both of which would already have degree $2$. 
This is impossible.
Consequently, the link of $2$ must be completed by the matching $\{35,46\}$, and hence 
\[235,246\in C.\]

The remaining edge degrees are then uniquely completed by
\[345,456\in C.\] 
Therefore, $C=\{123,124,125,126,134,156,235,246,345,456\}$,
which is of type $10$.
\item Suppose that $e=15$. Then $2e=3f$, so every edge is contained in
exactly two triangles. Since every edge of the complete graph on the six vertices occurs, each vertex link is a $5$-cycle. 
Thus, $K_C$ is a triangulation of a closed surface with $\chi(K_C)=6-15+10=1$. 
This implies that $|K_C|$ is not orientable.
Indeed, it is well known that closed orientable surfaces have even Euler characteristic. 

On the other hand, the circuit relation
\eqref{e:cirrel}
would imply that this surface is orientable. Indeed, if two triangles $T$ and $T'$ share an edge, 
then by comparing the coefficient of that edge, we have $|\lambda_T|=|\lambda_{T'}|$. 
Since the triangle-adjacency graph of $K_C$ is connected, all $|\lambda_T|$ are equal. 
After reversing the orientations of some triangles, the relation gives a coherent orientation of all triangles of $K_C$. 
Hence, $K_C$ would be orientable, a contradiction. 

Therefore, $e=15$ is impossible. 
\end{itemize}

Finally, suppose that $f=11$. Then \eqref{eq:ef} gives $e=15$, so every edge of the complete graph on the six vertices occurs. Moreover, 
\[\sum_E \bigl(m(E)-2\bigr)=3f-2e=33-30=3.\]
Since $2\leq m(E)\leq 4$ for every edge $E$, there are two possibilities for the degrees of the edges: 
\begin{itemize}
\item one edge has degree $4$, one edge has degree $3$, and all the other edges have degree $2$,
\item three edges have degree $3$ and all the other edges have degree $2$.
\end{itemize}
The first case is impossible by Lemma~\ref{lem:shakehand}, since there would be exactly one edge of odd degree.
Consequently, exactly three edges have degree $3$, while the remaining twelve edges have degree $2$. 
By Lemma~\ref{lem:shakehand}, the graph formed by the three odd-degree edges has even degree at every vertex. 
The only simple graph with three edges having this property is a triangle. Up to relabeling, we may assume that these three edges are $12, 13, 23$.

We first claim that $123\in C$. Suppose otherwise. Since each of $12,13,23$ has degree $3$, all nine triangles 
\[124,125,126,\qquad134,135,136,\qquad234,235,236\]
would have to belong to $C$. At this point, the edges $12,13,23$ would already have degree $3$, and every edge between $\{1,2,3\}$ and $\{4,5,6\}$ would already have degree $2$. 
Therefore, neither of the remaining two triangles could contain a vertex of $\{1,2,3\}$. 
However, the only triangle on $\{4,5,6\}$ is $456$, so there cannot be two further distinct triangles, a contradiction. Hence $123\in C$.

Let $L=\operatorname{lk}_{K_C}(1)$.
Then its vertex set is $\{2,3,4,5,6\}$, and the vertex degrees are given by
\[\deg_L(2)=\deg_L(3)=3,\qquad\deg_L(4)=\deg_L(5)=\deg_L(6)=2.\]
Since $123\in C$, the edge $23$ belongs to $L$. Deleting this edge reduces the degrees of $2$ and $3$ by one, 
and hence every vertex of $L\setminus\{23\}$ has degree $2$. 
Thus $L\setminus\{23\}$ is a $5$-cycle. Up to permuting $4,5,6$ and interchanging $2$ and $3$, we may assume that this cycle has edge set
\[\{24,34,36,56,25\}.\]
It follows that the triangles of $C$ containing $1$ are precisely \[123,124,125,134,136,156.\]

Next, let $L'=\operatorname{lk}_{K_C}(2)$. The triangles $123,124,125$ give the edges $13,14,15$ of $L'$. 
In particular, the vertex $1$ has already attained its required degree $3$ in $L'$, so no further edge of $L'$ is incident with $1$. 
Thus, the remaining edges of $L'$ lie on $\{3,4,5,6\}$, and the remaining degrees to be supplied at $3,4,5,6$ are $2,1,1,2$, respectively. 
Hence, the remaining three edges form a path of length three whose endpoints are $4$ and $5$ and whose internal vertices are $3$ and $6$. 
There are exactly two possibilities: $\{34,36,56\}$ or $\{35,36,46\}$. 

First, suppose that the former occurs. Then $234,236,256\in C$. At this point, every edge incident with $1$ or $2$ has attained its prescribed degree, so the remaining two triangles must be supported on $\{3,4,5,6\}$. 
However, each of the four triangles $345, 346, 356, 456$ contains at least one of the edges $34,36,56$, all of which already have degree $2$. This is impossible.

Therefore, the second possibility must occur, and hence \[235,236,246\in C.\]
The remaining two triangles must again be supported on $\{3,4,5,6\}$. 
Moreover, they cannot contain the edge $36$, since this edge already lies in the two triangles $136$ and $236$. 
Among the four triangles on $\{3,4,5,6\}$, the only ones not containing $36$ are 
\[345\quad\text{and}\quad 456.\]
Thus, both of them must belong to $C$.
We conclude that
\[
C=\{123,124,125,134,136,156,235,236,246,345,456\},
\]
which is of type $11$.

\medskip

It follows from the case analysis above that, up to relabeling of the vertices,
every circuit $C$ is one of the sets listed in the table.
We can also directly check that the listed sets are indeed circuits. 
\end{proof}

We now determine which pointed circuits $(C,N)$ satisfy the condition
$N\notin\operatorname{cl}(C\setminus\{N\})$.

\begin{lem}\label{lem:pointed-circuit-check}
Let $V=\{1,2,3,4,5,6\}$, and let $(C,N)$ be a pointed circuit in $\binom{V}{3}$
such that $N\notin\operatorname{cl}(C\setminus\{N\})$.
Then, up to relabeling of the vertices, $(C,N)$ is one of the following two candidates:
\begin{enumerate}
\item $C=\{123,124,135,145,236,246,356,456\}$ (type 8b) and $N=123$.
\item $C=\{123,124,125,134,136,156,235,236,246,345,456\}$ (type 11) and $N=123$.
\end{enumerate}
\end{lem}

\begin{proof}
By Lemma~\ref{lem:circuit}, every circuit contained in $\binom{V}{3}$ is, up to relabeling of the vertices, one of the six types listed there.
For each type, we determine which choices of the distinguished triangle $N$ can satisfy
$N\notin\operatorname{cl}(C\setminus\{N\})$.

\underline{The case of type 4.}\ \ 
This circuit is the boundary of a tetrahedron.
Thus, for every $N\in C$, the set $C\setminus\{N\}$ consists of the other three triangles on the same four vertices.
Hence $N \in\operatorname{cl}(C\setminus\{N\})$.
Therefore there is no pointed circuit of type $4$ satisfying the required condition.

\underline{The case of type 6.}\ \ 
Up to relabeling, we may assume that $N=123$.
Set $S=C\setminus\{123\}$.
Since $235,245,345\in S$, we have $234\in\operatorname{cl}(S)$.
Then, since $124,134,234\in\operatorname{cl}(S)$, we obtain
$123\in\operatorname{cl}(S)$, a contradiction.

\underline{The case of type 8b.}\ \  
This circuit is the boundary of the octahedron.
By relabeling of vertices, we may assume that $N=123$.
This gives the first candidate.

\underline{The cases of types 8a, 10, 11.}\ \  
If $C$ is of type 8a, then $(34), (16)(25) \in \operatorname{Aut}(C)$, so we may assume that
\[ N \in \{123, 134, 235\}.\]
If $C$ is of type 10, then $(36)(45), (12)(45) \in \operatorname{Aut}(C)$, so we may assume that 
\[ N \in \{ 123, 124, 134, 345\}.\]
If $C$ is of type 11, then $(123)(456), (12)(45) \in \operatorname{Aut}(C)$, so we may assume that
\[ N \in \{123, 124, 156, 456\}.\]
We now show that all of these cases, except the case of type $11$ with
$N=123$, do not satisfy the required condition.
Put $S=C\setminus\{N\}$.
In the following table, the notation $Q\Rightarrow T$ means that, on the
four-element set $Q$, the other three triangles are already contained in
$\operatorname{cl}(S)$; hence $T$ is also contained in $\operatorname{cl}(S)$.
The third column records the closure steps needed to find the circuit $D$
shown in the last column, which satisfies the following properties:
\begin{itemize}
\item $N\in D$,
\item $D$ is of type $6$ up to relabeling of the vertices,
\item $D\setminus\{N\}\subset\operatorname{cl}(S)$.
\end{itemize}
\begingroup
\footnotesize
\vspace{0.3\baselineskip}
\[
\begin{array}{c@{\qquad}c@{\qquad}l@{\qquad}l}
\toprule
\text{type} & N & \text{closure steps} & \text{circuit }D \\
\midrule
8\mathrm a & 123 &
\{3,4,5,6\}\Rightarrow345 &
\{123,124,134,235,245,345\}
\\[1mm]

8\mathrm a & 134 &
\{3,4,5,6\}\Rightarrow345 &
\{123,124,134,235,245,345\}
\\[1mm]

8\mathrm a & 235 &
\{1,2,3,4\}\Rightarrow234 &
\{234,235,245,346,356,456\}
\\[1mm]

10 & 123 &
\{1,2,4,6\}\Rightarrow146,\ \{1,4,5,6\}\Rightarrow145 &
\{123,125,134,145,235,345\}
\\[1mm]

10 & 124 &
\{1,2,3,5\}\Rightarrow135,\ \{1,3,4,5\}\Rightarrow145 &
\{124,126,145,156,246,456\}
\\[1mm]

10 & 134 &
\{1,2,3,5\}\Rightarrow135,\ \{1,2,4,6\}\Rightarrow146 &
\{134,135,146,156,345,456\}
\\[1mm]

10 & 345 &
\{1,2,3,4\}\Rightarrow234,\ \{1,2,5,6\}\Rightarrow256 &
\{234,235,246,256,345,456\}
\\[1mm]

11 & 124 &
\{1,2,3,5\}\Rightarrow135,\ \{1,2,3,6\}\Rightarrow126,\ 
\{1,3,4,5\}\Rightarrow145 &
\{124,126,145,156,246,456\}
\\[1mm]

11 & 156 &
\{1,2,3,4\}\Rightarrow234,\ \{1,2,3,5\}\Rightarrow135,\ 
\{2,3,4,6\}\Rightarrow346 &
\{135,136,156,345,346,456\}
\\[1mm]

11 & 456 &
\{1,2,3,4\}\Rightarrow234,\ \{1,2,3,5\}\Rightarrow135,\ 
\{1,3,5,6\}\Rightarrow356 &
\{234,236,246,345,356,456\}
\\
\bottomrule
\end{array}
\]
\vspace{0\baselineskip}
\endgroup

\noindent
By the same argument as in the case of type $6$, for each row, we have $N\in\operatorname{cl}(D\setminus\{N\})$.
Since $D\setminus\{N\}\subset\operatorname{cl}(S)$ and
$\operatorname{cl}(S)$ is tetrahedrally closed, we have
$\operatorname{cl}(D\setminus\{N\})\subset\operatorname{cl}(S)$.
Thus $N\in\operatorname{cl}(S)$, a contradiction.
Consequently, among the cases of types 8a, 10, and 11, the only case which is not excluded is the case of type 11 with $N=123$.
This gives the second candidate.
\end{proof}

The following lemma analyzes the first candidate obtained in
Lemma~\ref{lem:pointed-circuit-check}.

\begin{lem}\label{lem:first-candidate}
Let $H=(V,\mathcal H)$ be a $3$-uniform hypergraph with
$V=\{1,2,3,4,5,6\}$ satisfying the tetrahedron condition.
Suppose that
\[
S:=\{124,135,145,236,246,356,456\}\subset\mathcal H
\quad\text{and}\quad
123\notin\mathcal H.
\]
Then $\mathcal H=S$.  In particular, $H$ is an octahedral obstruction.
\end{lem}

\begin{proof}
Set $C:=S\cup\{123\}
=\{123,124,135,145,236,246,356,456\}$.
Then $C$ is the boundary of the octahedron.
Its opposite pairs are $\{1,6\}, \{2,5\}, \{3,4\}$,
and the triangles in $C$ are precisely the triangles which contain one vertex from each of these three pairs.

It is enough to show that no triangle in $\binom{V}{3}\setminus C$ belongs to $\mathcal H$.
Let $M\in\binom{V}{3}\setminus C$.
Since $M$ contains one of the opposite pairs and
the automorphism group of the pointed circuit $(C,123)$ acts transitively on the three opposite pairs, we may assume that $M$ contains $\{1,6\}$.
Then the third vertex of $M$ belongs to $\{2,3,4,5\}$.
Since $(23)(45)\in\operatorname{Aut}(C,123)$, we may assume that this third vertex is either $2$ or $5$.
Hence, without loss of generality, we may assume that $M=126$ or $M=156$.
We use the notation $Q\Rightarrow T$ for closure steps as in the proof of Lemma~\ref{lem:pointed-circuit-check}.

First suppose that $M=126$.
Starting from $S\cup\{126\}$, we obtain
\[
\{1,2,4,6\}\Rightarrow146,\qquad
\{1,4,5,6\}\Rightarrow156,\qquad
\{1,3,5,6\}\Rightarrow136,\qquad
\{1,2,3,6\}\Rightarrow123.
\]
Thus $123\in\operatorname{cl}(S\cup\{126\})$.

Next suppose that $M=156$.
Starting from $S\cup\{156\}$, we obtain
\[
\{1,4,5,6\}\Rightarrow146,\qquad
\{1,2,4,6\}\Rightarrow126,\qquad
\{1,3,5,6\}\Rightarrow136,\qquad
\{1,2,3,6\}\Rightarrow123.
\]
Thus $123\in\operatorname{cl}(S\cup\{156\})$.

Therefore, if some $M\in\binom{V}{3}\setminus C$ belongs to $\mathcal H$, then the tetrahedron condition implies
$123\in\operatorname{cl}(S\cup\{M\})\subset\mathcal H$,
contrary to the assumption $123\notin\mathcal H$. 
Hence  $\mathcal H=S$.
Therefore $H$ is an octahedral obstruction.
\end{proof}

The following lemma analyzes the second candidate obtained in Lemma~\ref{lem:pointed-circuit-check}.

\begin{lem}\label{lem:second-candidate}
Let $H=(V,\mathcal H)$ be a $3$-uniform hypergraph with
$V=\{1,2,3,4,5,6\}$ satisfying the tetrahedron condition.
Suppose that
\[
S:=\{124,125,134,136,156,235,236,246,345,456\}\subset\mathcal H
\quad\text{and}\quad
123\notin\mathcal H.
\]
Then $\mathcal H=S$, and $H$ is a point hypergraph.
\end{lem}

\begin{proof}
Set $C:=S\cup\{123\}
=\{123,124,125,134,136,156,235,236,246,345,456\}$.
It is enough to show that no triangle in $\binom{V}{3}\setminus C$ belongs to $\mathcal H$.
Let $M\in\binom{V}{3}\setminus C$.
The subgroup of $\operatorname{Aut}(C,123)$ generated by $(123)(456)$ and
$(12)(45)$ has two orbits on $\binom{V}{3}\setminus C$, namely
\[
\{126,135,234\}
\quad\text{and}\quad
\{256,346,145,146,245,356\}.
\]
Hence it is enough to consider the two cases
$M=126$ and $M=256$.

First suppose that $M=126$.
Starting from $S\cup\{126\}$, we obtain
\[
\{1,2,3,6\}\Rightarrow123.
\]
Thus $123\in\operatorname{cl}(S\cup\{126\})$.

Next suppose that $M=256$.
Starting from $S\cup\{256\}$, we obtain
\[
\{1,2,5,6\}\Rightarrow126,\qquad
\{1,2,3,6\}\Rightarrow123.
\]
Thus $123\in\operatorname{cl}(S\cup\{256\})$.

Therefore, if some $M\in\binom{V}{3}\setminus C$ belongs to $\mathcal H$, then
the tetrahedron condition yields
$123\in\operatorname{cl}(S\cup\{M\})\subset\mathcal H$,
contrary to the assumption $123\notin\mathcal H$.  Hence $\mathcal H=S$.

To see that $H$ is a point hypergraph, it is enough to find one choice of parameters. Set
\[
a_{23}=a_{26}=a_{35}=a_{45}=a_{46}=-1,
\]
and $a_{ij}=1$ for all remaining pairs $1\le i<j\le6$.
Also set $a_{ji}=a_{ij}^{-1}$.
Then a direct computation gives
\[
a_{ij}a_{jk}a_{ki}
=
\begin{cases}
1 & \{i,j,k\}\in S,\\
-1 & \{i,j,k\}\notin S,
\end{cases}
\]
for every $i<j<k$. Therefore $H$ is a point hypergraph.
\end{proof}

Note that the set $S$ in Lemma~\ref{lem:second-candidate}
is, up to relabeling, the triangle set of the six-vertex
triangulation of $\mathbb{RP}^2$ in
\cite[Example~6.1.3]{CL}.

We are now ready to prove the main theorem of this section.

\begin{thm}\label{thm:tetraocta}
Let $H=(V,\mathcal H)$ be a $3$-uniform hypergraph with $|V|=6$.
Then $H$ is a point hypergraph if and only if $H$ satisfies the tetrahedron condition and is
not an octahedral obstruction.
\end{thm}

\begin{proof}
The only-if direction is Proposition~\ref{prop:tetraocta}.
We prove the converse.

Suppose that $H$ satisfies the tetrahedron condition and is not an octahedral obstruction.
Assume, to the contrary, that $H$ is not a point hypergraph. 
By Lemma~\ref{lem:rational-point-criterion}, there exists a non-face $N\notin\mathcal H$ such that
\[
[\partial\vv{N}]_H\otimes 1=0
\quad\text{in }G_H\otimes_{\mathbb Z}\mathbb Q.
\]
By Lemma~\ref{lem:bad-nonface}, there exists a pointed circuit $(C,N)$ such that
$C\setminus\{N\}\subset\mathcal H$
and $N\notin\operatorname{cl}(C\setminus\{N\})$.
By Lemma~\ref{lem:pointed-circuit-check}, up to relabeling of the vertices,
$(C,N)$ is one of the two candidates listed there.
If $(C,N)$ is the first candidate, then Lemma~\ref{lem:first-candidate} implies that $H$ is an octahedral obstruction. This contradicts the assumption.
If $(C,N)$ is the second candidate, then Lemma~\ref{lem:second-candidate}
implies that $H$ is a point hypergraph.
This also contradicts our assumption.
Therefore we conclude that $H$ is a point hypergraph.
\end{proof}

\begin{rem}\label{rem:bdl-six-vertices}
Combining the count of $175$ isomorphism classes of adequate
collections in \cite{BDL} with Theorem~\ref{thm:tetraocta}, we obtain
exactly $174$ point hypergraphs on six vertices up to isomorphism.
By \cite[Lemma~2.14]{HU}, this is also the number of isomorphism
classes of point schemes of six-variable skew polynomial algebras.

The octahedral obstruction is the single exception needed in
the six-variable case of the realization statement in
\cite[Theorem~7]{BDL}.
To see this, consider the octahedral obstruction
$H=(V,\mathcal H)$ with
$V=\{1,2,3,4,5,6\}$ and
$\mathcal H=\{124,135,145,236,246,356,456\}$.
Let $\mathcal C$ be the collection of coordinate planes indexed
by the non-faces of $H$.
It is adequate by the tetrahedron condition.
It is also dense, since $126,136,146,156$ are all non-faces and hence
\[
\#\{\mathbb P(\{1,6,k\})\in\mathcal C\mid k=2,3,4,5\}
=4\ge 3.
\]
However, $H$ is non-point by Theorem~\ref{thm:surface-obst}.
\end{rem}

\section*{Acknowledgments}
The authors thank Pieter Belmans for helpful comments and suggestions that improved the exposition of this paper.
The first author was supported by JSPS KAKENHI Grant Numbers JP24K00521 and JP24K00534. The second author was supported by JSPS KAKENHI Grant Number JP26K06761.

\section*{Declaration on the use of AI}
OpenAI's ChatGPT (GPT-6 Astra) was used as an aid in checking mathematical arguments, including finite computations, and in improving the organization, exposition, and wording of the manuscript.
All suggestions and outputs were independently reviewed and verified by the authors, who take full responsibility for the final content of the paper.

\end{document}